\documentclass[11pt]{article}
\usepackage[tbtags]{amsmath}
\usepackage{amssymb}
\usepackage{amsthm}
\usepackage[misc]{ifsym}
\usepackage{cases}
\usepackage{mathrsfs}
\usepackage{color}
\usepackage{xcolor}

\numberwithin{equation}{section}   

\normalsize

\title{\bf  Relationship between MP and DPP for Risk-Sensitive Stochastic Optimal Control Problems: Viscosity Solution Framework 
\thanks{This work is financially supported by National Natural Science Foundations of China (12471419, 12271304), and Shandong Provincial Natural Science Foundation (ZR2024ZD35).}}
\author{\normalsize Huanqing Dong\thanks{\it School of Mathematics, Shandong University, Jinan 250100, P.R. China, E-mail: donghuanqing@mail.sdu.edu.cn},
 \quad Jingtao Shi\thanks{\it Corresponding author, School of Mathematics, Shandong University, Jinan 250100, P.R. China, E-mail: shijingtao@sdu.edu.cn}}

\newtheorem{mythm}{Theorem}[section]
\newtheorem{mydef}{Definition}[section]
\newtheorem{mylem}{Lemma}[section]
\newtheorem{Remark}{Remark}[section]

\begin{document}
	
\maketitle

\noindent{\bf Abstract:}\quad 
In this paper, we study the relationship between general maximum principle and dynamic programming principle for risk-sensitive stochastic optimal control problems, where the control domain is not necessarily convex. 
The original problem is equivalent to a stochastic recursive optimal control problem of a forward-backward system with quadratic generators. 
Relations among the adjoint processes, the generalized Hamiltonian function and the value function are proved under the framework of viscosity solutions. 
Some examples are given to illustrate the theoretical results.

\vspace{2mm}
	
	
\vspace{2mm}
	
	
\section{Introduction}

It is well-known that Pontryagin's \emph{maximum principle} (MP) and Bellman's \emph{dynamic programming principle} (DPP) are two of the most important tools in solving optimal control problems. 
So a natural question is: Are there any connections between these two methods? The relationship between MP and DPP is studied in many literatures. 
For the deterministic optimal control problems, the results were discussed by Fleming and Rishel \cite{FR75}, Barron and Jensen \cite{BJ86} and Zhou \cite{Z90-1}. 
Bensoussan \cite{B82} (see also \cite{YZ99}) obtained the relationship between MP and DPP for stochastic optimal control problems.  Zhou \cite{Z90-2,Z91} (see also \cite{YZ99}) generalized their relationship using the viscosity solution theory, without the assumption that the value function is smooth.

In this paper, we consider the relationship between MP and DPP for a risk-sensitive stochastic optimal control problem, where the system is given by the following controlled \emph{stochastic differential equation} (SDE):
\begin{equation}\label{state SDE}
\left\{
\begin{aligned}
	dX(t) &= b(t, X(t), u(t)) dt + \sigma(t, X(t), u(t)) dW(t),\\
	X(0) &= x,
\end{aligned}
\right.
\end{equation}
and the cost/objective functional is defined by
\begin{equation}\label{risk-sensitive cost functional}
	J(u(\cdot)) = \mu^{-1}\log \mathbb{E} \left[\exp\left\{\mu\int_0^T f(t, X(t), u(t))  dt+ \mu h(X(T))\right\}\right].
\end{equation}
Risk-sensitive optimal control problems have attracted many researchers since the early work of Jacobson \cite{J73}. 
On the one hand, this problem can be reduced to the risk-neutral case under the limit of the risk-sensitive parameter. 
On the other hand, it has close connection with differential games \cite{J92,LZ01}, robust control theory \cite{LZ01,MB17}, and can be applied to describe the risk attitude of an investor by the risk-sensitive parameter in mathematical finance \cite{BP99,FS00}.

Whittle \cite{W90,W91} first derived Pontryagin's MP based on the theory of large deviation for a risk-sensitive stochastic optimal control problem. 
A new risk-sensitive maximum principle was established by Lim and Zhou \cite{LZ05} under the condition that the value function is smooth, based on the general stochastic MP of Peng \cite{P90} and the relationship between MP and DPP of Yong and Zhou \cite{YZ99}. 
This idea is used in Shi and Wu \cite{SW11,SW12} for studying jump processes case and mathematical finance. 
Djehiche et al. \cite{DTT15}  extend the results of \cite{LZ05} to risk-sensitive control problems for dynamics that are non-Markovian and mean-field type. 
Further developments on risk-sensitive stochastic control problems in a diverse set of settings, see  Moon et al. \cite{MDB19}, Chala \cite{C17-1,C17-2}, Lin and Shi \cite{LS25-1,LS25-2} and the references therein.

In El Karoui and Hamad\`{e}ne \cite{EH03}, the risk-sensitive objective functional is characterized by a {\it backward SDE} (BSDE) with quadratic generators that does not necessarily satisfy the usual Lipschitz condition. Consider the following BSDE with quadratic generators:
\begin{equation}\label{BSDE quadratic growth}
\left\{
\begin{aligned}
	dY(t) &= -\left[f(t, X(t), Y(t), Z(t), u(t)) + \frac{\mu}{2}|Z(t)|^2\right]dt +  Z(t)dW(t),\\
	Y(T) &= h(X(T)).
\end{aligned}
\right.
\end{equation}
By the DPP approach, Moon \cite{M21} extended the results of classical risk-sensitive optimal control problem studied in \cite{FS06} of the objective functional given by the BSDE like \eqref{BSDE quadratic growth}. 
They showed that the value function is the unique viscosity solution of the corresponding Hamilton-Jacobi-Bellman equation under an additional risk-sensitive condition. 
When $f$ is independent of $y$ and $z$, it is easy to check that $Y(0)=\mu^{-1}\log \mathbb{E} \left[\exp\left\{\mu\int_{0}^{T} f(t, X(t), u(t))  \, dt+ \mu h(X(T))\right\}\right]$. 
Thus, the risk-sensitive optimal control problem of minimizing $J(u(\cdot))$ in \eqref{risk-sensitive cost functional} subject to \eqref{state SDE} is equivalent to minimizing $Y(0)$ subject to \eqref{state SDE} and \eqref{BSDE quadratic growth}. 
In this case, the controlled system becomes a decoupled \emph{forward-backward SDE} (FBSDE), with the objective functional of a BSDE with quadratic generators. 
Kobylanski \cite{K00} first studied the solvability of quadratic BSDEs with the bounded terminal values in 2000. 
Briand and Hu \cite{BH06,BH08} studied the solvability of quadratic BSDEs with unbounded terminal and convex generator. 
Hu et al. \cite{HJX22} obtained a global maximum principle for decoupled forward-backward stochastic control problems with quadratic generators. 
For more research on  quadratic BSDEs, see \cite{Z17,FT18,HT16,LS25-2} and the references therein.

For the relationship between MP and DPP for optimal control problems of decoupled FBSDEs/recursive utilities, Shi and Yu \cite{S13} investigated the local case in which the control domain is convex and the value function is smooth. 
Nie et al. \cite{NSW17} studied the general case by the second-order adjoint equation of Hu \cite{Hu17}, within the framework of viscosity solutions. 
Hu et al. \cite{HJX20} studied the relationship between general MP and DPP for the fully coupled forward-backward stochastic controlled system within the framework of viscosity solutions. 
Note that all the aforementioned references considered stochastic optimal control problems for BSDE objective functionals with linear growth coefficients satisfying Lipschitz conditions. 
For the relationship between MP and DPP for optimal control problems of decoupled FBSDEs and BSDE objective functionals with quadratic generators, very recently, Wu et al. \cite{W25} established the connection between MP and DPP for the risk-sensitive recursive utility singular control problem under the assumption of smooth value function, where the cost functional is given by a BSDE with quadratic  generators, driven by a discontinuous semimartingale. 
Dong and Shi \cite{DS26} researched the relationship between MP and DPP for risk-sensitive stochastic optimal control problems under the smooth assumption of the value functions. 
However, to the best of our knowledge, there are no results about the relationship between MP and DPP for risk-sensitive stochastic optimal control problems under the framework of viscosity solutions.

Inspired by the above works, in this paper, we will derive the relationship between MP and DPP for risk-sensitive stochastic optimal control problems. 
As mentioned above, the risk-sensitive criteria are connected to quadratic BSDEs. 
Hence, the original problem can be transformed to a stochastic recursive optimal control problem of a forward-backward system with quadratic generators. 
For this problem, because the cost functional \eqref{risk-sensitive cost functional} can be defined by the controlled BSDE with quadratic generators, we wish to connect the MP of Hu et al. \cite{HJX22} under certain conditions. 
And we wish to apply the DPP of Moon \cite{M21} under certain conditions. 
We obtain that the connection between the adjoint process $(p,P)$ in the maximum principle and the first-order and second-order sub- (resp. super-) jets of the value function $V$ in the $x$-variable is
\begin{equation}\label{relationship 1}
\begin{aligned}
	\{p(s)\} \times [P(s),\infty) \subseteq D_x^{2,+} V(s,\bar{X}^{t,x;\bar{u}}(s)), && \forall s \in [t,T],\ \mathbf{P}\text{-a.s.}, \\
	D_x^{2,-} V(s,\bar{X}^{t,x;\bar{u}}(s)) \subseteq \{p(s)\} \times (-\infty,P(s)], && \forall s \in [t,T],\ \mathbf{P}\text{-a.s.},
\end{aligned}
\end{equation}
and the connection between the function $\mathcal{H}_{1}$ and the right sub- (resp. super-) jets of $V$ in the $t$-variable is
\begin{equation}\label{relationship 2}
\begin{aligned}
	[-\mathcal{H}_{1}(s,\bar{X}^{t,x;\bar{u}}(s),\bar{Z}^{t,x;\bar{u}}(s),\infty) \subseteq D_{t^+}^{1,+} V(s,\bar{X}^{t,x;\bar{u}}(s)),  &&\text{a.e.}\, s \in [t,T],\ \mathbf{P}\text{-a.s.}, \\
	D_{t^+}^{1,-} V(s,\bar{X}^{t,x;\bar{u}}(s)) \subseteq (-\infty,-\mathcal{H}_{1}(s,\bar{X}^{t,x;\bar{u}}(s),\bar{Z}^{t,x;\bar{u}}(s),\bar{u}(s))], && \text{a.e.}\, s \in [t,T],\ \mathbf{P}\text{-a.s.}.
\end{aligned}
\end{equation}

The main contributions of this paper can be summarized as follows.

(1) We study the relationship between MP and DPP for risk-sensitive stochastic optimal control problems within the framework of viscosity solutions, where the control domain is not necessarily convex. 
The original problem can be transformed to a stochastic recursive optimal control problem of a forward-backward system with quadratic generators. 
We extend the work of Zhou \cite{Z90-2,Z91} to the risk-sensitive optimal control problem and partially generalize to the work by Nie et al. \cite{NSW17} to a forward-backward system, where the backward equation has quadratic growth.

(2) Compared with \cite{NSW17} and \cite{HJX20}, the difficulties of proving the above relations come from the quadratic growth property of our controlled system. 
Due to the quadratic growth property, when we establish the relation \eqref{relationship 1}, we need to perturb the initial state $x$ which leads to the decoupled variational equation \eqref{recursive FBSDE1}, where the backward equations are linear BSDEs with unbounded stochastic Lipschitz coefficients involving BMO-martingales, which has been studied by Briand and Confortola \cite{BC08} (see \cite{AID07,HJX22,FT18} for some extensions). 

(3) Inspired by Hu et al. \cite{HJX20}, we utilize the relationship between $(\hat{Y}(\cdot),\hat{Z}(\cdot))$ and $\hat{X}(\cdot)$ (see equation \eqref{relationship between X,Y and Z}) and estimate the remainder terms of the forward and backward equations simultaneously. 
Then we obtain a class of linear BSDEs with unbounded random Lipschitz coefficients (see equation \eqref{Stochastic Lipschitz BSDE1} and \eqref{Stochastic Lipschitz BSDE2}), for which corresponding well-posedness results are available. 
Consequently, we obtain the required estimation results, thereby enabling the establishment of relationship \eqref{relationship 1}. The proof idea for relationship \eqref{relationship 2} is similar to this.

The rest of this paper is organized as follows. 
In Section 2, we state our risk-sensitive stochastic optimal control problem and give some preliminaries. 
In Section 3, we state MP and DPP for risk-sensitive stochastic optimal control. 
In Section 4, we give the connections between the value function and the adjoint processes within the framework of the viscosity solutions. 
We also give two interesting examples to explain the related results in Section 5. 
Finally, in Section 6, we give the concluding remarks.
Proofs of the main results are given in the Appendix.

$\mathit{Notations}$. In this paper, we denote by $\mathbb{R}^n$ the space of $n$-dimensional Euclidean space, by $\mathbb{R}^{n \times d}$ the space of $n \times d$ matrices, and by $\mathbb{S}^n $ the space of $n \times n$ symmetric matrices. 
For $x \in \mathbb{R}^n$, $x^\top$ denotes its transpose.  $\langle \cdot , \cdot \rangle$ and $| \cdot |$ denote the scalar product and norm in the Euclidean space, respectively. The trace of a matrix A is $\operatorname{tr}(A)$. 
Let $C([0,T] \times \mathbb{R}^n)$ be the set of real-valued continuous functions defined on $[0,T] \times \mathbb{R}^n$, $C^{1,k}([0,T] \times \mathbb{R}^n)$ ($k \geq 1$) be the set of real-valued functions such that $f \in C([0,T] \times \mathbb{R}^n)$, $f_t$ and the partial derivatives of $f$ with respect to $x$ up to the $k$th order are continuous and bounded. 

\section{Problem statement and preliminaries}

Let $T > 0$ be fixed, and $U \subset \mathbb{R}^m$ be nonempty and compact. Given $t \in [0, T )$, denote by $\mathcal{U}^{w}[t,T]$ the set of all 5-tuples $(\Omega, \mathcal{F}, \mathbf{P}, W(\cdot); u(\cdot))$
satisfying the following conditions:
(i) $(\Omega, \mathcal{F}, \mathbf{P})$ is a complete probability space; 
(ii) $\{W(s)\}_{s \geq t}$ is a one-dimensional standard Brownian motion defined on $(\Omega, \mathcal{F}, \mathbf{P})$ over $[t,T]$ (with $W(t)=0$, a.s.), and
$\mathcal{F}^t_s=\sigma\{W(r); t\leq r \leq s\}$ augmented by all the $\mathbf{P}$-null sets in $\mathcal{F}$;
and (iii) $u(\cdot): [t, T] \times \Omega \to U$ is an $\{\mathcal{F}^t_s\}_{s \geq t}$-adapted process on $(\Omega, \mathcal{F}, \mathbf{P})$ such that $\mathbb{E}\left[ \int_{t}^{T} |u(s)|^p \, \mathrm{d}s \right]$ for any $p>0$.
	
For any given $p,q \geq 1$, we introduce the following spaces.

$\mathcal{M}^{p,q}_{\mathcal{F}}([t,T];\mathbb{R}^n)$: the space of $\{\mathcal{F}^t_s\}_{s \geq t}$-adapted processes $\psi(\cdot)$ on $[t,T]$ satisfying 
\begin{equation*}
	\left\| \psi \right\|_{p,q} 
	:= \left( \mathbb{E}\left[ \left( \int_{t}^{T} |\psi(s)|^p \, ds
	\right)^{\frac{q}{p}} \right] \right)^{\frac{1}{q}} < \infty.
\end{equation*}
In particular, we denote by $M^{p}_{\mathcal{F}}([t,T];\mathbb{R}^n)$ the above space when $p=q$.
	
$L^\infty_{\mathcal{F}}([t,T];\mathbb{R}^n)$: the space of $\{\mathcal{F}^t_s\}_{s \geq t}$-adapted processes $\psi(\cdot)$ on $[t,T]$ satisfying 
\begin{equation*}
	\left\| \psi \right\|_\infty := \underset{(s,\omega) \in [t,T] \times \Omega}{\mathrm{ess\,sup}} |\psi(s,\omega)|< \infty.
\end{equation*}

$\mathcal{S}^{p}_{\mathcal{F}}([t,T];\mathbb{R}^n)$: the space of $\{\mathcal{F}^t_s\}_{s \geq t}$-adapted processes $\psi(\cdot)$ on $[t,T]$ satisfying 
\begin{equation*}
	\mathbb{E}\left[\sup\limits_{s \in [t,T]}|\psi(s)|^{p}\right] < \infty.
\end{equation*}

$BMO_p$:  the space of $\{\mathcal{F}^t_s\}_{s \geq t}$-adapted and real-valued martingale $M$ satisfying
\begin{equation*}
	\left\| M \right\|_{BMO_p} := \underset{\tau}{\mathrm{sup}} \left\| \left( \mathbb{E}\left[ |M_T - M_\tau|^p \mid \mathcal{F}_\tau \right] \right)^{\frac{1}{p}} \right\|_\infty < \infty, \quad p \in [1, +\infty),
\end{equation*}
where the supremum is taken over all stopping times $\tau \in [t,T]$.	

In the following lemma we state the properties and notations of BMO-martingales, we refer readers to \cite{K94} or \cite{HY92} and the references therein for more details.
\begin{mylem}\label{properties of BMO martingales}
(1) Let $ p \in (1, +\infty) $. Then, there is a positive constant $ C_p $ such that, for any $\{\mathcal{F}^t_s\}_{s \geq t}$-adapted and real-valued martingales $ M $, we have
\begin{equation*}
\|M\|_{BMO_1} \leq \|M\|_{BMO_p} \leq C_p \|M\|_{BMO_1}.
\end{equation*}
Thus, for any \( p \geq 1 \), we write simply $BMO$ for $BMO_p$.

(2) Denote by $\mathcal{E}(M)$ the Dol\'eans–Dade exponential of a continuous local martingale $M$, that is, $\mathcal{E}(M_s) = \exp\left\{ M_s - \frac{1}{2}\langle M \rangle_s \right\}$ for any $s \in [t,T]$. If $M \in BMO$, then $\mathcal{E}(M)$ is a uniformly integrable martingale.

(3) The reverse H\"older inequality: Let $M\in BMO$. Denote by $p_*$ the positive constant linked to a BMO-martingale $M$ such that $\Psi(p_*)=\|M\|_{BMO_{2}}$, where the monotonically decreasing function 
\begin{equation*}
\Psi(x)=\sqrt{1+\frac{1}{x^{2}}\ln\frac{2x-1}{2(x-1)}}-1,\quad x\in(1,+\infty).
\end{equation*}
We make the convention that $\Psi(+\infty):=\lim_{x\to+\infty}\Psi(x)=0$ and $p_{*}=+\infty$ when $\|M\|_{BMO_{2}}=0$. Thus $p_*$ is uniquely determined. If $p\in(1,p_*)$, then, for any stopping time $\tau\in[0,T]$,
\begin{equation*}
\mathbb{E}\left[(\mathcal{E}(M_{T}))^{p}\mid\mathcal{F}_{\tau}\right]\leq K\left(p,\|M\|_{BMO_{2}}\right)(\mathcal{E}(M_{\tau}))^{p},\quad \mathbf{P}\text{-a.s.},
\end{equation*}
where the constant $K\left(p,\|M\|_{BMO_{2}}\right)$ can be chosen depending only on $p$ and $\|M\|_{BMO_{2}}$, e.g.,
\begin{equation*}
K\left(p,\|M\|_{BMO_{2}}\right)=2\left(1-\frac{2p-2}{2p-1}\exp\left\{p^{2}\left[\|M\|_{BMO_{2}}^{2}+2\|M\|_{BMO_{2}}\right]\right\}\right)^{-1}.
\end{equation*}
By $q_{*}$ we denote the conjugate exponent of $p_{*}$, that is, $(p_{*})^{-1} + (q_{*})^{-1} = 1$.

(4) $H \cdot W$: $H$ is an $\{\mathcal{F}^t_s\}_{s \geq t}$-adapted process and $H \cdot W$ is the stochastic integral of $H$ with respect to $W$, that is, $H \cdot W = \int_t^{\cdot} H(s)dW(s)$.
\end{mylem}

\subsection{Problem formulation}
	
For any initial time and state $(t,x) \in [0,T) \times \mathbb{R}^n$ and $u(\cdot) \in \mathcal{U}^{w}[t, T]$, consider the state $X^{t,x;u}(\cdot) \in \mathbb{R}^n$ given by the following controlled
$\mathrm{SDE}$:
\begin{equation}\label{controlled state SDE}
\left\{
\begin{aligned}
	dX^{t, x; u}(s) &= b(s, X^{t, x; u}(s), u(s)) ds + \sigma(s, X^{t, x; u}(s), u(s)) dW(s),\\
	X^{t, x; u}(t) &= x,\quad s \in [t,T]. 
\end{aligned}
\right.
\end{equation}
Here $b: [t,T] \times \mathbb{R}^n \times U \to \mathbb{R}^{n}$, $\sigma: [t,T] \times \mathbb{R}^n \times U \to \mathbb{R}^{n}$ are given functions.

We make the following assumption.\\
$\mathbf{Assumption\ 1}$ $b,\sigma$ are uniformly continuous in $(s,x,u)$, and there exists a constant $L_1>0$ such that for any $s \in [t,T]$, $x_1,x_2 \in \mathbb{R}^n, u \in U$,
\begin{equation}\label{assumption 1}
\begin{cases}
	|b(s,x_1,u)-b(s,x_2,u)| + |\sigma(s,x_1,u)-\sigma(s,x_2,u)| \leq L_1|x_1-x_2|,\\
	|b(s,0,u)|+ |\sigma(s,0,u)| \leq L_1(1+|u|).
\end{cases}
\end{equation}

For any $u(\cdot) \in \mathcal{U}^{w}[t, T]$, under Assumption 1, the equation \eqref{controlled state SDE} has a unique solution $X^{t, x; u}(s) \in \mathcal{S}^{p}_{\mathcal{F}}([t,T];\mathbb{R}^n)$
by the classical SDE theory. We consider the following cost functional with initial condition  $(t,x) \in [0,T) \times \mathbb{R}^n$ and $u(\cdot) \in \mathcal{U}^{w}[t,T]$:
\begin{equation}\label{risk-sensitive cost functional 2}
\begin{aligned}
	J(t,x;u(\cdot) )
	= \mu^{-1}\log \mathbb{E} \left[\exp\Big\{\mu\int_t^T f(s,X^{t, x; u}(s),u(s)) ds+ \mu h(X^{t, x; u}(T))\Big\}\right],
\end{aligned}
\end{equation}
where  $f: [t,T] \times \mathbb{R}^n \times U \to \mathbb{R}$, $h: \mathbb{R}^n \to \mathbb{R}$, and $\mu>0$ is the risk-sensitive parameter, which implies that the controller is risk-averse. 
In fact, let $J_1 \equiv \int_0^T f(s,X^{0, x; u}(s),u(s)) ds+ h(X^{0, x; u}(T))$. When the risk-sensitive parameter $\mu$ is small, the cost functional \eqref{risk-sensitive cost functional 2} can be expanded as
\begin{equation*}
	\mathbb{E}[J_1] + \frac{\mu}{2}\operatorname{Var}(J_1) + O(\mu^2),
\end{equation*}
where $\operatorname{Var}(J_1)$ denotes the variance of $J_1$. If $\mu<0$, which implies that the controller is risk-seeking from an economic point of view.  If $\mu \to 0$, it is reduced to risk-neutral case.

Our risk-sensitive stochastic optimal control problem is the following.

{\bf Problem (RS)}. For given $(t,x) \in [0,T) \times \mathbb{R}^n$, we aim to minimize \eqref{risk-sensitive cost functional 2} subject to \eqref{controlled state SDE} over $\mathcal{U}^{w}[t,T]$.

\subsection{Equivalent stochastic recursive optimal control problem}

From  \cite{EH03} (see also \cite{M21}), we note that minimizing \eqref{risk-sensitive cost functional 2} in Problem (RS) is equivalent to minimize the following cost functional
\begin{equation}\label{recursive utility}
	J(t,x;u(\cdot)) = Y^{t,x;u}(t),
\end{equation}
where $Y^{t,x;u}(\cdot)$ is the first component of the solution pair $(Y^{t,x;u}(\cdot),Z^{t,x;u}(\cdot))$ to the scalar-valued BSDE with quadratic generator:
\begin{equation}\label{scalar-valued BSDE}
\left\{
\begin{aligned}
	dY^{t, x; u}(s) &= -\left[f(s, X^{t, x; u}(s), u(s)) + \frac{\mu}{2}|Z^{t,x;u}(s)|^2\right]ds +  Z^{t,x;u}(s)dW(s),\\
	Y^{t, x; u}(T) &= h(X^{t, x; u}(T)).
\end{aligned}
\right.
\end{equation}
So our Problem (RS) is equivalent to a stochastic recursive optimal control problem, where the cost functional is \eqref{recursive utility} and the state equation is given by the following controlled FBSDE:
\begin{equation}\label{recursive FBSDE}
\begin{cases}
\begin{aligned}
	&dX^{t, x; u}(s)=b(s,X^{t, x; u}(s),u(s))ds +\sigma(s,X^{t, x; u}(s),u(s))dW(s),\\
	&dY^{t, x; u}(s) = -\left[f(s, X^{t, x; u}(s), u(s)) + \frac{\mu}{2}|Z^{t,x;u}(s)|^2\right]ds +  Z^{t,x;u}(s)dW(s),\\
	&X^{t, x; u}(t)=x,\ Y^{t, x; u}(T) = h(X^{t, x; u}(T)).
\end{aligned}
\end{cases}
\end{equation}

We make the following assumption.\\
$\mathbf{Assumption\ 2}$  $f$ is uniformly continuous in $(s,x,u)$ and bounded, $h$ is uniformly continuous in $x$ and bounded. For $\phi=f,h$, there exists a constant $L_2>0$, such that for any $s \in [t,T], x_1, x_2 \in \mathbb{R}^n$, and $u \in U$,
\begin{equation}\label{assumption H2}
		|\phi(s,x_1,u)-\phi(s,x_2,u)| \leq L_2|x_1-x_2|.
\end{equation}
Note that Assumption 2 implies $|f(s, x, u)+\frac{\mu}{2}|z|^2| \leq L(1 + |x| + |z|^2)$. 
Thus \eqref{scalar-valued BSDE} can be viewed as a controlled quadratic BSDE. 
From the existence result (Proposition 3) in \cite{BH08} and the uniqueness result (Lemma 2.1) in \cite{HT16} ( or see Theorem 2.3 in \cite{HJX22} ), we have the following result.

\begin{mylem}\label{solvability of recursive FBSDE}
Let Assumption 1 and 2 hold. Then, for any $u(\cdot) \in \mathcal{U}^{w}[t,T]$ and $p > 1$, 
the state equation \eqref{recursive FBSDE} admits a unique solution $(X^{t, x; u}(\cdot), Y^{t, x; u}(\cdot), Z^{t, x; u}(\cdot)) \in \mathcal{S}_{\mathcal{F}}^p([t,T];
\mathbb{R}^n) \times L_{\mathcal{F}}^{\infty}([t,T];\mathbb{R}) \times \mathcal{M}_{\mathcal{F}}^{2,p}([t,T];\mathbb{R})$ such that $Z^{t, x; u} \cdot W \in BMO$.  Furthermore, for any fixed
$s\in[t,T]$, we have the following estimates:
\begin{equation}
\begin{aligned}
	&\mathbb{E}\left[\sup\limits_{r \in [s,T]}|X^{t, x; u}(r)|^{p}\big|\mathcal{F}_s^t\right]\\
    &\quad \leq C_1 \left\{ |X^{t, x; u}(s)|^p + \mathbb{E}\left[ \left( \int_{s}^T |b(r,0,u(r))| dr \right)^p + \left( \int_{s}^T |\sigma(r,0,u(r))|^2 dr \right)^{\frac{p}{2}}\Big|\mathcal{F}_s^t\right] \right\}, \\
	&\|Y^{t, x; u}\|_{\infty} + \|Z^{t, x; u} \cdot W\|_{BMO_2} \leq C_2,
\end{aligned}
\end{equation}
where $C_1$ depends on $(p, T, L_1)$, and $C_2$ depends on  $(\mu, T, \|f\|_{\infty},\|h\|_{\infty})$.
\end{mylem}

In the spirit of Lemma 3.3 in \cite{FT18}, we give a priori estimate of BSDE \eqref{scalar-valued BSDE}. 

\begin{mylem}\label{priori estimate of BSDE}
For any $u(\cdot) \in \mathcal{U}^{w}[t, T]$, and given unique solution $X^{t, x; u}(\cdot) \in \mathcal{S}_{\mathcal{F}}^p([t,T];
\mathbb{R}^n)$, let $( Y^{t, x; u}_{i}(\cdot), Z^{t, x; u}_{i}(\cdot)), i=1,2$, denote the solution to the following quadratic BSDE: 
\begin{equation}
	Y^{t, x; u}_{i}(s) = h_{i}(X^{t, x; u}(T))+ \int_s^T\left[f_{i}(r, X^{t, x; u}(r), u(r)) + \frac{\mu}{2}|Z^{t,x;u}_{i}(r)|^2\right]dr - \int_s^T Z^{t,x;u}_{i}(r)dW(r),
\end{equation}
where  $f_i: [t,T] \times \mathbb{R}^n \times U \to \mathbb{R}$ and $h_i: \mathbb{R}^n \to \mathbb{R}$ are bounded and measurable.
Then for $p \geq 2$ and any fixed $s\in[t,T]$, we have the following estimate:
\begin{equation}\label{priori estimate of quadratic BSDE}
\begin{aligned}
	\mathbb{E}\bigg[\,& \sup_{r \in [s,T]} \bigl| Y^{t,x;u}_{1}(r) - Y^{t,x;u}_{2}(r) \bigr|^{p} 
	+ \Bigl( \int_{s}^T \bigl| Z^{t,x;u}_{1}(r) - Z^{t,x;u}_{2}(r) \bigr|^2 dr \Bigr)^{\frac{p}{2}} \Big| \mathcal{F}_s^t \bigg] \\
	&\leq C_3 \Bigg\{ \mathbb{E}\Bigg[ |h_{1}(X^{t,x;u}(T)) - h_{2}(X^{t,x;u}(T))|^{p\bar{q}^2} \\
	&\qquad\quad + \Bigl( \int_{s}^T \bigl| f_{1}(r, X^{t,x;u}(r), u(r)) - f_{2}(r, X^{t,x;u}(r), u(r)) \bigr| dr \Bigr)^{p\bar{q}^2} \Big| \mathcal{F}_s^t \Bigg] \Bigg\}^{\frac{1}{\bar{q}^2}},
\end{aligned}
\end{equation}
where a positive constant $\bar{q}$ satisfying  $q_*\leq\bar{q}<\infty$, and the lower bound $q_*>1$ controlled only by $(\mu, T, \|f_{1}\|_{\infty}, \|f_{2}\|_{\infty},
\|h_{1}\|_{\infty}, \|h_{2}\|_{\infty} )$, and some positive constant $C_3$ depending only on $(p, \bar{q}, \mu, T, \|f_{1}\|_{\infty}, \|f_{2}\|_{\infty}, \|h_{1}\|_{\infty}, \|h_{2}\|_{\infty} )$.

In particular, taking $h_1=0$ and $f_1=0$, we have
\begin{equation*}
\begin{aligned}
	\mathbb{E}\bigg[\,& \sup_{r \in [s,T]} \bigl| Y^{t,x;u}_{2}(r) \bigr|^{p} 
	+ \Bigl( \int_{s}^T \bigl| Z^{t,x;u}_{2}(r) \bigr|^2 dr \Bigr)^{\frac{p}{2}} \Big| \mathcal{F}_s^t \bigg] \\
	&\leq C_3 \Bigg\{ \mathbb{E}\Bigg[ | h_{2}(X^{t,x;u}(T))|^{p\bar{q}^2} 
	+ \Bigl( \int_{s}^T \bigl| f_{2}(r, X^{t,x;u}(r), u(r)) \bigr| dr \Bigr)^{p\bar{q}^2} \Big| \mathcal{F}_s^t \Bigg] \Bigg\}^{\frac{1}{\bar{q}^2}}.
\end{aligned}
\end{equation*}
\end{mylem}

\begin{proof}
Denote
\begin{equation*}
\begin{aligned}
	&\delta Y^{t,x;u}(s)=Y^{t,x;u}_{1}(s)-Y^{t,x;u}_{2}(s),\ \delta Z^{t,x;u}(s)=Z^{t,x;u}_{1}(s)-Z^{t,x;u}_{2}(s),\\
	&\delta f(s)=(f_{1}-f_{2})(s, X^{t,x;u}(s), u(s)),\ \delta h(T)=(h_{1}-h_{2})(X^{t,x;u}(T)).
\end{aligned}
\end{equation*}
Then we use a linearization method and rewrite
\begin{equation}\label{BSDE1}
	\delta Y^{t,x;u}(s)=\delta h(T) + \int_s^T \left(\delta f(r) +a(r)\delta Z^{t,x;u}(r)\right) dr -  \int_s^T \delta Z^{t,x;u}(r)dW(r),
\end{equation}
where the processes $a(r)$ are defined by 
\begin{equation*}
	a(r)=\frac{\frac{\mu}{2}\left(|Z^{t,x;u}_{1}(r)|^2-|Z^{t,x;u}_{2}(r)|^2\right)}{|\delta Z^{t,x;u}(r)|^2}\mathbf{1}_{\delta Z^{t,x;u}(r) \neq 0} \delta Z^{t,x;u}(r),
\end{equation*}
which implies that $|a(r)| \leq \mu\left(|Z^{t,x;u}_{1}(r)|+|Z^{t,x;u}_{2}(r)|\right)$. 
By Lemma \ref{solvability of recursive FBSDE}, $Z^{t,x;u}_{1} \cdot W$ and $Z^{t,x;u}_{2} \cdot W$ belong to $BMO$, so that $a \cdot W$ belongs to $BMO$. 
Since $h_i$ and $f_i$ are bounded, $\delta h(T)$ and $\delta f(r)$ are bounded. Thus, the BSDE \eqref{BSDE1} satisfies Assumption 3 of Section 2.3. 
By Lemma 2.4, for all $p \geq 2$ and any fixed $s\in[t,T]$, we have 
\begin{equation*}\begin{aligned}
	&\mathbb{E}\left[\,\sup_{s\leq r\leq T}|\delta Y^{t,x;u}(r)|^{p} 
	+ \Bigl(\int_{s}^T|\delta Z^{t,x;u}(r)|^2\,dr\Bigr)^{\frac{p}{2}}  \Big| \mathcal{F}_s^t \right]\\
	&\leq C_4\Bigg(\mathbb{E}\left[|\delta h(T)|^{p\bar{q}^2}
	+ \Bigl(\int_{s}^T|\delta f(r)|\,dr\Bigr)^{p\bar{q}^2} \Big| \mathcal{F}_s^t \right]\Bigg)^{\frac{1}{\bar{q}^2}},
\end{aligned}\end{equation*}
where $\bar{q}$ is a positive constant satisfying $q_*\leq\bar{q}<\infty$ with lower bound $q_*>1$ controlled only by $\|a \cdot W\|_{BMO_2}$ , and $C_4$ is a positive constant depending only on $(p,\bar{q},T,\|a \cdot W\|_{BMO_2})$. 
By Lemma 2.2, $\|a \cdot W\|_{BMO_2}$ is bounded by some constant depending only on $(\mu, T, \|f_{1}\|_{\infty}, \|f_{2}\|_{\infty}, \|h_{1}\|_{\infty},\\ \|h_{2}\|_{\infty})$.
Thus, \eqref{priori estimate of quadratic BSDE} holds. The proof is complete.
\end{proof}

We now define the value function as
\begin{equation}\label{value function}
	V(t,x):=\inf \limits_{u(\cdot) \in\, \mathcal{U}^{w}[t,T]}J(t,x;u(\cdot)),\quad(t,x) \in [0,T) \times \mathbb{R}^n.
\end{equation}
Any $\bar{u}(\cdot) \in \mathcal{U}^{w}[t,T]$ satisfying \eqref{value function} is called an optimal control, and the corresponding solution $(\bar{X}^{t,x;\bar{u}}(\cdot),\bar{Y}^{t,x;\bar{u}}(\cdot),\bar{Z}^{t,x;\bar{u}}(\cdot))$ to \eqref{recursive FBSDE} is called optimal state trajectory.

\begin{Remark}\label{V is meaningful}
From \cite{M21}, we know that under Assumption 1 and 2, the above value function is a deterministic function, so our definition \eqref{value function} is meaningful.
\end{Remark} 
 
\subsection{Stochastic Lipschitz BSDEs}

Due to the quadratic growth feature of the BSDE. Some BSDE with stochastic Lipschitz coefficients involving BMO-martingales appear in variational equations. 
We need the result of the existence and the uniqueness of solutions to this kind equations and their estimates. Consider the following BSDE:
\begin{equation}\label{Stochastic Lipschitz BSDE}
\left\{
\begin{aligned}
	-dY(s) &= g(s,Y(s),Z(s))ds - Z(s)dW(s),\\
	Y(T) &= \xi,
\end{aligned}
\right.
\end{equation}
where $g: \Omega \times [t,T] \times \mathbb{R} \times \mathbb{R} \to \mathbb{R}$, $\xi: \Omega \to \mathbb{R}$.

We make the following assumption.\\
$\mathbf{Assumption\ 3}$ The map $(\omega, s)\to g(\omega, s, \cdot, \cdot)$ is $\{\mathcal{F}_s^t\}_{s\ge t}$-progressively measurable.\\
(1) There exist a positive constant $K$ and a positive $\{\mathcal{F}_s^t\}_{s\ge t}$-progressively measurable process $H\cdot W\in BMO$ such that, for every $(y,z), (y',z') \in \mathbb{R} \times \mathbb{R}$,
\begin{equation*}
	|g(\omega,s,y,z) - g(\omega,s,y',z')| \leq K|y - y'| + H_s(\omega)|z - z'|, \quad \text{a.e. } s \in [t,T],\ \mathbf{P}\text{-a.s.}.
\end{equation*}
(2) $\xi$ is $\mathcal{F}_T$-measurable and for $\forall p\ge 2$, $\mathbb{E}\left[|\xi|^p + \left(\int_t^T |g(s,0,0)|ds\right)^p\right] < \infty.$

We note that Ankirchner et al. \cite{AID07} and Briand and Confortola \cite{BC08} studied a class of BSDEs with stochastic Lipschitz coefficients involving BMO-martingales. 
Fujii and Takahashi \cite{FT18} generalized the results of \cite{AID07} and \cite{BC08}, extending their results of locally Lipschitz BSDEs with BMO coefficients to the jump-diffusion setting, and derived a priori estimates for such BSDEs (see their Theorem A.1). 
If we consider the without jumps, the following results can be obtained.

\begin{mylem}\label{solvability of Stochastic Lipschitz BSDEs}
Let Assumption 3 hold. Then there exists a unique solution $(Y, Z)$ to the BSDE \eqref{Stochastic Lipschitz BSDE}. Moreover, for all $p \geq 2$  and any fixed $s\in[t,T]$, we have 
\begin{equation*}\label{estimate of stochastic Lipschitz BSDE}
  \mathbb{E}\left[\,\sup_{s\leq r\leq T}|Y(r)|^{p} + \Bigl(\int_s^T|Z(r)|^2\,dr\Bigr)^{\frac{p}{2}}  \Big| \mathcal{F}_s^t \right]
  \leq C_4\Bigg(\mathbb{E}\left[|\xi|^{p\bar{q}^2}+ \Bigl(\int_s^T|g(r,0,0)|\,dr\Bigr)^{p\bar{q}^2} \Big| \mathcal{F}_s^t \right]\Bigg)^{\frac{1}{\bar{q}^2}},
\end{equation*}
where $\bar{q}$ is a positive constant satisfying $q_*\leq\bar{q}<\infty$ with lower bound $q_*>1$ controlled only by $\|H \cdot W\|_{BMO_2}$, and $C_4$ is a positive constant depending only on $(p,\bar{q},T,K,\|H \cdot W\|_{BMO_2})$.
\end{mylem}

\section{MP and DPP}

In the following, we introduce the risk-sensitive DPP and MP approaches for Problem(RS) in the literatures, respectively. 
Moon \cite{M21} studied a generalized risk-sensitive optimal control problem, where the objective functional is defined by the controlled quadratic BSDEs. 
A generalized risk-sensitive DPP for the value function was obtained. And the corresponding value function is a viscosity solution to the HJB equation. Under an additional parameter condition, the viscosity solution is unique. 
When $l$ is independent of $y$ and $z$ in (3) of \cite{M21}, the BSDE in (3) is reduced to the BSDE in \eqref{scalar-valued BSDE}. 
Thus, the DPP for Problem (RS) follows as a specialization of the result in \cite{M21}. For this purpose, we first define the family of backward semigroups associated with the BSDE \eqref{scalar-valued BSDE}. The notion of stochastic
backward semigroup was first introduced by Peng \cite{P97}.

Given the initial data $(t, x)$, a positive number $\tau \leq T - t$ and a real-valued random variable $b \in L^\infty(\Omega,\mathcal{F}_{t+\tau},\mathbf{P};\mathbb{R})$, we define the backward semigroup
\begin{equation*}
	\Pi_{s,t+\tau}^{t,x;u}[b] := \tilde{Y}^{t,x;u}(s), \qquad s \in [t, t + \tau],
\end{equation*}
where $(\tilde{Y}^{t,x;u}(s), \tilde{Z}^{t,x;u}(s))$ is the solution of the following BSDE with the time horizon $t + \tau$:
\begin{equation*}
\left\{
\begin{aligned}
	&d\tilde{Y}^{t,x;u}(s) 
	=  -\left[f(s, X^{t, x; u}(s), u(s)) + \frac{\mu}{2}|\tilde{Z}^{t,x;u}(s)|^2\right]ds +  \tilde{Z}^{t,x;u}(s)dW(s), \quad s \in [t, t + \tau], \\
	&\tilde{Y}^{t,x;u}(t+\tau) = b,
\end{aligned}
\right.
\end{equation*}
where $X^{t,x;u}(\cdot)$ is the solution of SDE \eqref{controlled state SDE}.

We now state the risk-sensitive DPP, which is of \cite{M21}, Theorem 1.
\begin{mylem}
Consider Problem (RS). Suppose that Assumptions 1 and 2 hold. For any $t, t + \tau \in [0, T]$ with $t < t + \tau$, and for any $x \in \mathbb{R}^n$, the value function in \eqref{value function} satisfies the following DPP:
\begin{equation}
	V(t, x) = \inf_{u \in \mathcal{U}^{w}[t, t + \tau]} \Pi_{t, t + \tau}^{t, x; u} \big[ V(t + \tau,  X^{t, x; u}(t + \tau)) \big].
\end{equation}
\end{mylem}

Next we can state the HJB equation for Problem (RS):
\begin{equation}\label{HJB equation}
\begin{cases}
	v_t(t,x) + \inf\limits_{u \in U} G(t,x,u,v_x(t,x),v_{xx}(t,x))=0,\quad(t,x) \in [0,T) \times \mathbb{R}^n,\\
	v(T,x)=h(x),\quad x \in \mathbb{R}^n,
\end{cases}
\end{equation}
where the generalized Hamiltonian function $G:[0,T] \times \mathbb{R}^n \times U \times \mathbb{R}^n \times \mathbb{S}^n \to \mathbb{R}$ is defined by
\begin{equation}\label{generalized Hamiltonian}
\begin{aligned}
	G(t,x,u,p,P)&:= f(t,x,u) + \langle p ,b(t,x,u) \rangle  + \frac{\mu}{2}|\sigma(t,x,u)^{\top}p|^2 \\
	&\quad+ \frac{1}{2}\operatorname{tr}(\sigma(t,x,u)\sigma(t,x,u)^{\top}P).	
\end{aligned}	
\end{equation}
 
Now, we introduce the following definition of viscosity solution to HJB equation
\eqref{HJB equation}, which can be found in \cite{CIL92}.
\begin{mydef}
(i) A real-valued function $v \in C([0,T] \times \mathbb{R}^{n})$ is called a viscosity subsolution of \eqref{HJB equation} if $v(T,x) \leq h(x)$ for all $x \in \mathbb{R}^{n}$, and for all $\varphi \in C^{1,2}([0,T] \times \mathbb{R}^{n})$ such that $v-\varphi$ attains a local maximum at $(t,x) \in [0,T) \times \mathbb{R}^{n}$, we have
\begin{equation*}
\begin{aligned}
	\varphi_{t}(t,x) + \inf\limits_{u \in U} G(t,x, u ,\varphi_{x}(t,x),\varphi_{xx}(t,x)) \geq 0.	
\end{aligned}
\end{equation*}
(ii) A real-valued function $v \in C([0,T] \times \mathbb{R}^{n})$ is called a viscosity supersolution of \eqref{HJB equation} if $v(T,x) \geq h(x)$ for all $x \in \mathbb{R}^{n}$, and for all $\varphi \in C^{1,2}([0,T] \times \mathbb{R}^{n})$ such that $v-\varphi$ attains a local minimum at $(t,x) \in [0,T) \times \mathbb{R}^{n}$, we have
\begin{equation*}
\begin{aligned}
	\varphi_{t}(t,x) + \inf\limits_{u \in U} G(t,x, u ,\varphi_{x}(t,x),\varphi_{xx}(t,x)) \leq 0.	
\end{aligned}
\end{equation*}
(iii) A real-valued function $v \in C([0,T] \times \mathbb{R}^{n})$ is called a viscosity solution of \eqref{HJB equation} if it is both a viscosity subsolution and viscosity supersolution to \eqref{HJB equation}.
\end{mydef}

\begin{Remark}
The viscosity solution of the HJB equation \eqref{HJB equation} can be equivalently defined in the language of sub- and super-jets (see \cite{CIL92}).
\end{Remark}

The following result is of \cite{M21}, Theorem 2 and Proposition 1.
\begin{mylem}
Suppose that Assumption 1 and 2 hold. Then  the value function in \eqref{value function} is a viscosity solution to the HJB equation in \eqref{HJB equation}. Moreover, there exists a constant $\mu^*>0$ such that for $\mu \in [0, \mu^*)$, the value function is a unique viscosity solution to the HJB equation.
\end{mylem}

\begin{Remark} 
The parameter $\mu^*$ is closely related to break down and conjugate points in a class of risk-sensitive and stochastic differential games (see \cite{B99}, \cite{FS00}, \cite{MDB19}, \cite{N96}).
\end{Remark}

Hu et al. \cite{HJX22} studied a stochastic optimal control problem for forward-backward control systems with quadratic generators. A global stochastic maximum principle is deduced (Theorem 3.16 in \cite{HJX22}). If we replace $f(s,x,y,z,u)$ in (2.7) of \cite{HJX22} with $f(s, x, u) + \frac{\mu}{2}|z|^2$ so that the state equation is consistent with \eqref{recursive FBSDE}. Thus we can obtain the MP of our Problem (RS).

First, we need introduce the following assumption.\\
$\mathbf{Assumption\ 4}$ $b,\sigma,f,h$ are twice continuously differentiable with respect to $x$. Their first and second-order partial derivatives in $x$ are uniformly bounded, Lipschitz continuous in $(x,u)$.

Let $(\bar{X}^{t,x;\bar{u}}(\cdot),\bar{Y}^{t,x;\bar{u}}(\cdot),\bar{Z}^{t,x;\bar{u}}(\cdot),\bar{u}(\cdot))$ be an optimal quadruple. 
For any $s \in [t,T]$, we denote $\bar{b}(s):=b(s,\bar{X}^{t,x;\bar{u}}(s),\bar{u}(s))$, $\bar{\sigma}(s):=\sigma(s,\bar{X}^{t,x;\bar{u}}(s),\bar{u}(s))$, $\bar{f}(s):= f(s,\bar{X}^{t,x;\bar{u}}(s),\bar{u}(s))$, and similar notations used for all their derivatives. For given $(t,x) \in [0,T) \times \mathbb{R}^n$, we introduce the following first-order adjoint process $(p(\cdot),q(\cdot)) $ satisfying
\begin{equation}\label{the first-order adjoint equation}
\left\{
\begin{aligned}
	dp(s) &= -\bigg[\bar{b}_x(s)^{\top}p(s) + \bar{f}_x(s) + \bar{\sigma}_x(s)^{\top}q(s)\\
          &\qquad + \mu \bar{Z}^{t,x;\bar{u}}(s)\left(\bar{\sigma}_x(s)^{\top}p(s) + q(s)\right)\bigg]ds + q(s)dW(s),\\
	p(T) &= h_x(\bar{X}^{t,x;\bar{u}}(T)),
\end{aligned}
\right.
\end{equation}
and the second-order adjoint process $(P(\cdot),Q(\cdot)) $ satisfying
\begin{equation}\label{the second-order adjoint equation}
\left\{
\begin{aligned}
	dP(s) &= -\bigg[\bar{b}_x(s)^{\top}P(s) +P(s)\bar{b}_x(s) + \bar{\sigma}_x(s)^{\top}\left(P(s) + \mu p(s)p(s)^{\top}\right)\bar{\sigma}_x(s) \\ 
	&\qquad + \bar{\sigma}_x(s)^{\top}\left(Q(s) + \mu P(s)\bar{Z}^{t,x;\bar{u}}(s) + \mu p(s)q(s)^{\top}\right)\\ 
	&\qquad + \left(Q(s) + \mu P(s)\bar{Z}^{t,x;\bar{u}}(s) + \mu q(s)p(s)^{\top}\right)\bar{\sigma}_x(s) \\
	&\qquad +\mu q(s)q(s)^{\top} + \bar{H}_{xx}(s) + \mu Q(s)\bar{Z}^{t,x;\bar{u}}(s)\bigg]ds + Q(s)dW(s), \\
	P(T) &= h_{xx}(\bar{X}^{t,x;\bar{u}}(T)),
\end{aligned}
\right.
\end{equation}
where $\bar{H}(s):= H(s,\bar{X}^{t,x;\bar{u}}(s),\bar{Z}^{t,x;\bar{u}}(s),\bar{u}(s),p(s),q(s))$ with the Hamiltonian function $H$ defined by
\begin{equation}\label{risk-sensitive Hamiltonian}
	H(s,x,z,u,p,q):= \langle p ,b(s,x,u) \rangle + f(s,x,u) + q^\top\sigma(s,x,u) + \mu \sigma(s,x,u)^\top p z.	
\end{equation}

Then, we have the following MP for our Problem (RS) (see Hu et al. \cite{HJX22}).
\begin{mylem}\label{risk-sensitive MP}
Suppose that Assumptions 1, 2 and 4 hold and let $(t,x) \in [0,T) \times \mathbb{R}^n$ be fixed. 
Let  $(\bar{X}^{t,x;\bar{u}}(\cdot),\bar{Y}^{t,x;\bar{u}}(\cdot),\bar{Z}^{t,x;\bar{u}}(\cdot),\bar{u}(\cdot))$ be an optimal quadruple for our Problem (RS). 
Then for $p>1$, there exist unique solutions
\begin{equation*}
\left\{
\begin{aligned}
	&(p(\cdot),q(\cdot)) \in  L^\infty_{\mathcal{F}}([t,T];\mathbb{R}^n) \times \mathcal{M}^{2,p}_{\mathcal{F}}([t,T];\mathbb{R}^n),\\
	&(P(\cdot),Q(\cdot)) \in \mathcal{S}^{p}_{\mathcal{F}}([t,T];\mathbb{S}^n) \times \mathcal{M}^{2,p}_{\mathcal{F}}([t,T];\mathbb{S}^n),
\end{aligned}
\right.
\end{equation*}
to the first-order adjoint equation \eqref{the first-order adjoint equation} and the second-order adjoint equation \eqref{the second-order adjoint equation}, respectively, such that the following inequality holds: 
\begin{equation}\label{maximum condition-mathcal H}
\begin{aligned}
	&\mathcal{H}(s,\bar{X}^{t,x;\bar{u}}(s),\bar{Z}^{t,x;\bar{u}}(s),\bar{u}(s),p(s),q(s),P(s))\\ 
    &\leq \mathcal{H}(s,\bar{X}^{t,x;\bar{u}}(s),\bar{Z}^{t,x;\bar{u}}(s),u,p(s),q(s),P(s)),\quad\forall u \in U, \ \text{a.e.} s \in [t,T],\ \mathbf{P}\text{-a.s.}, 
\end{aligned}	
\end{equation}
where
\begin{equation}\label{mathcal H}
\begin{aligned}
	\mathcal{H}(s,x,z,u,p,q,P)&:= \langle p ,b(s,x,u) \rangle + f(s,x,u) + q^\top\sigma(s,x,u) + \mu (\sigma(s,x,u)-\bar{\sigma}(s))^\top p z+\frac{\mu}{2}|z|^2\\
	&\quad+\frac{1}{2}\left[\left(\sigma(s,x,u)-\bar{\sigma}(s)\right)^\top\left(P(s) + \mu p(s)p(s)^\top\right)\left(\sigma(s,x,u)-\bar{\sigma}(s)\right)\right].
\end{aligned}	
\end{equation}
\end{mylem}

\section{Main result}

In this subsection, we investigate the relationship between the MP and the DPP. We first recall the notion of the right parabolic superjet of a continuous function on $[0,T)\times\mathbb{R}^{n}$.
For $v\in C\left([0,T]\times\mathbb{R}^{n}\right)$ and $(\hat{t},\hat{x})\in[0,T)\times\mathbb{R}^{n}$, the right parabolic superjet of $v$ at $(\hat{t},\hat{x})$ is the set triple
\begin{equation*}
\begin{aligned}
	D_{t+,x}^{1,2,+} v(\hat{t},\hat{x}) :=& \biggl\{(q,p,P) \in \mathbb{R} \times \mathbb{R}^n \times \mathbb{S}^{n} | v(t,x) \leq v(\hat{t},\hat{x}) + q(t-\hat{t}) +  \langle p,x-\hat{x} \rangle \\
	&\quad+ \frac{1}{2}(x-\hat{x})^{\top}P(x-\hat{x}) + o\left(|t-\hat{t}|+ |x-\hat{x}|^2 \right), \text{ as } t \downarrow \hat{t}, x \to \hat{x} \biggl\},
\end{aligned}
\end{equation*}
and the right parabolic subjet of $v$ at $(\hat{t},\hat{x})$ is the set triple
\begin{equation*}
\begin{aligned}
	D_{t+,x}^{1,2,-} v(\hat{t},\hat{x}) :=& \biggl\{(q,p,P) \in \mathbb{R} \times \mathbb{R}^n \times \mathbb{S}^{n} | v(t,x) \geq v(\hat{t},\hat{x}) + q(t-\hat{t}) +  \langle p,x-\hat{x} \rangle \\
	&\quad+ \frac{1}{2}(x-\hat{x})^{\top}P(x-\hat{x}) + o\left(|t-\hat{t}|+ |x-\hat{x}|^2 \right), \text{ as } t \downarrow \hat{t}, x \to \hat{x} \biggl\}.
\end{aligned}
\end{equation*}
From the above definitions, we see immediately that
\begin{equation*}
\left\{\begin{array}{l}
	D_{t+,x}^{1,2,+}v(\hat{t},\hat{x})+[0,\infty)\times\{0\}\times\mathbb{S}^{n}_{+}=D_{t+,x}^{1,2,+}v(\hat{t},\hat{x}),\\
	D_{t+,x}^{1,2,-}v(\hat{t},\hat{x})-[0,\infty)\times\{0\}\times\mathbb{S}^{n}_{+}=D_{t+,x}^{1,2,-}v(\hat{t},\hat{x}),
\end{array}\right.
\end{equation*}
where $\mathbb{S}^{n}_{+}:=\{S \in \mathbb{S}^{n}|S\geq0\}$, and $A\pm B:=\{ a \pm b \mid a \in A,\ b \in B \}$ for any subsets \(A\) and \(B\) in a same Euclidean space.

We will also make use of the partial super/sub-jets with respect to one of the variables $t$ and $x$. Therefore, we need the following definitions. 
\begin{equation}\label{second-order super- and sub-jets in the spatial variable}
\left\{
\begin{aligned}
	D_x^{2,+} v(\hat{t},\hat{x}) &:= \biggl\{(p,P) \in \mathbb{R}^n \times \mathbb{S}^{n} | v(\hat{t},x) \leq v(\hat{t},\hat{x}) + \langle p,x-\hat{x} \rangle \\
	&\qquad+ \frac{1}{2}(x-\hat{x})^{\top}P(x-\hat{x}) + o\left( |x-\hat{x}|^2 \right), \text{ as } x \to \hat{x} \biggl\}, \\
	D_x^{2,-} v(\hat{t},\hat{x}) &:= \biggl\{(p,P) \in \mathbb{R}^n \times \mathbb{S}^{n} | v(\hat{t},x) \geq v(\hat{t},\hat{x}) + \langle p,x-\hat{x} \rangle \\
	&\qquad+ \frac{1}{2}(x-\hat{x})^{\top}P(x-\hat{x}) + o\left( |x-\hat{x}|^2 \right), \text{ as } x \to \hat{x}\biggl\},
\end{aligned}
\right.
\end{equation}
\begin{equation}
\left\{
\begin{aligned}
	D_x^{1,+} v(\hat{t},\hat{x}) &:= \biggl\{p \in \mathbb{R}^n| v(\hat{t},x) \leq v(\hat{t},\hat{x}) + \langle p,x-\hat{x} \rangle +o\left( |x-\hat{x}| \right), \text{ as } x \to \hat{x} \biggl\}, \\
	D_x^{1,-} v(\hat{t},\hat{x}) &:= \biggl\{p\in \mathbb{R}^n | v(\hat{t},x) \geq v(\hat{t},\hat{x}) + \langle p,x-\hat{x} \rangle + o\left( |x-\hat{x}| \right), \text{ as } x \to \hat{x}\biggl\},
\end{aligned}
\right.
\end{equation}
and
\begin{equation}\label{right super- and sub-jets in the time variable}
\left\{
\begin{aligned}
	D_{t+}^{1,+}v(\hat{t}, \hat{x}) & := \left\{q \in \mathbb{R} : v(t, \hat{x}) \leq v(\hat{t}, \hat{x}) + q(t - \hat{t}) 
	+ o\left(|t - \hat{t}|\right) \text{ as } t \downarrow \hat{t}\right\},\\
	D_{t+}^{1,-}v(\hat{t}, \hat{x}) & := \left\{q \in \mathbb{R} : v(t, \hat{x}) \geq v(\hat{t}, \hat{x}) + q(t - \hat{t}) 
	+ o\left(|t - \hat{t}|\right) \text{ as } t \downarrow \hat{t}\right\}.
\end{aligned}
\right.
\end{equation}

The following theorem shows that the adjoint variables $p$, $P$ and the value function $V$ are connected within the framework of the superjet and the subjet in the state variable $x$ along an optimal trajectory.
The proof is left in the Appendix.

\begin{mythm}\label{relation of MP and DPP in spatial variable}
Suppose Assumptions 1, 2 and 4 hold and let $(t,x) \in [0,T) \times \mathbb{R}^{n}$ be fixed. 
Suppose that $\bar{u}(\cdot)$ is an optimal control for Problem (RS), and $(\bar{X}^{t,x;\bar{u}}(\cdot),\bar{Y}^{t,x;\bar{u}}(\cdot),\bar{Z}^{t,x;\bar{u}}(\cdot))$ is the optimal trajectory. 
Let $(p(\cdot),q(\cdot)) \in \bigcap_{p > 1} \left( L^\infty_{\mathcal{F}}([t,T];\mathbb{R}^n) \times \mathcal{M}_\mathcal{F}^{2, p}([t, T]; \mathbb{R}^n) \right)$ 
and $(P(\cdot),Q(\cdot)) \in \bigcap_{p > 1} \left( \mathcal{S}_\mathcal{F}^p([t, T]; \mathbb{S}^n) \times \mathcal{M}_\mathcal{F}^{2, p}([t, T]; \mathbb{S}^n) \right)$ be the first-order and second-order adjoint process satisfying \eqref{the first-order adjoint equation} and \eqref{the second-order adjoint equation}, respectively. Then
\begin{equation}\label{relation of adjoint and second-order super and sub-jets in spatial variable}
\begin{aligned}
	\{p(s)\} \times [P(s),\infty) \subseteq D_x^{2,+} V(s,\bar{X}^{t,x;\bar{u}}(s)), && \forall s \in [t,T],\ \mathbf{P}\text{-a.s.}, \\
	D_x^{2,-} V(s,\bar{X}^{t,x;\bar{u}}(s)) \subseteq \{p(s)\} \times (-\infty,P(s)], && \forall s \in [t,T],\ \mathbf{P}\text{-a.s.}.
\end{aligned}
\end{equation}
We also have
\begin{equation}\label{relation of adjoint and first-order super and sub-jets in spatial variable}
\begin{aligned}
	D_x^{1,-}V(s,\bar{X}^{t,x;\bar{u}}(s)) \subseteq \{p(s)\} \subseteq D_x^{1,+} V(s,\bar{X}^{t,x;\bar{u}}(s)), && \forall s \in [t,T],\ \mathbf{P}\text{-a.s.}.
\end{aligned}
\end{equation}
\end{mythm}

The following result characterizes the super- and subjets of the value function in the time variable $t$ along an optimal trajectory, with the help of an additional $\mathcal{H}_1$-function.
The proof is also left in the Appendix.

\begin{mythm}\label{relation of MP and DPP in time variable}
Suppose Assumptions 1, 2 and 4 hold and let $(t,x) \in [0,T) \times \mathbb{R}^{n}$ be fixed. 
Suppose that $\bar{u}(\cdot)$ is an optimal control for Problem (RS), and $(\bar{X}^{t,x;\bar{u}}(\cdot),\bar{Y}^{t,x;\bar{u}}(\cdot),\bar{Z}^{t,x;\bar{u}}(\cdot))$ is the optimal trajectory. 
Let  $(p(\cdot),q(\cdot)) \in \bigcap_{p > 1} \left( L^\infty_{\mathcal{F}}([t,T];\mathbb{R}^n) \times \mathcal{M}_\mathcal{F}^{2, p}([t, T]; \mathbb{R}^n) \right)$ and $(P(\cdot),Q(\cdot)) \in \bigcap_{p > 1} \left( \mathcal{S}_\mathcal{F}^p([t, T]; \mathbb{S}^n) \times \mathcal{M}_\mathcal{F}^{2, p}([t, T]; \mathbb{S}^n) \right)$ be the first-order and second-order adjoint process satisfying \eqref{the first-order adjoint equation} and \eqref{the second-order adjoint equation}, respectively. Then
\begin{equation}\label{relationship of adjoint and super- and sub-jets in time variable}
\begin{aligned}
	[-\mathcal{H}_{1}(s,\bar{X}^{t,x;\bar{u}}(s),\bar{Z}^{t,x;\bar{u}}(s),\bar{u}(s)),\infty) \subseteq D_{t^+}^{1,+} V(s,\bar{X}^{t,x;\bar{u}}(s)), &&\text{a.e.}\, s \in [t,T],\ \mathbf{P}\text{-a.s.}, \\
	D_{t^+}^{1,-} V(s,\bar{X}^{t,x;\bar{u}}(s)) \subseteq (-\infty,-\mathcal{H}_{1}(s,\bar{X}^{t,x;\bar{u}}(s),\bar{Z}^{t,x;\bar{u}}(s),\bar{u}(s))], && \text{a.e.}\, s \in [t,T],\ \mathbf{P}\text{-a.s.},
\end{aligned}
\end{equation}
where $\mathcal{H}_{1} : [t,T] \times \mathbb{R}^n \times \mathbb{R} \times U \to \mathbb{R}$ is defined as
\begin{equation}\label{mathcal H1}
\begin{aligned}
	\mathcal{H}_{1}(s,x,z,u) 
	&:=\mathcal{H}(s,x,z,u,p(s),q(s),P(s)) - \frac{1}{2}\bar{\sigma}(s)^\top P(s)\bar{\sigma}(s)
\end{aligned}
\end{equation}
\end{mythm}

\begin{Remark}
Since our controlled system is decoupled, our initial idea is to follow the approach proposed by Nie et al. \cite{NSW17}. 
In this case, we will also obtain a class of linear BSDEs corresponding to equation (4.13) in \cite{NSW17}. 
This is a linear BSDE with unbounded random Lipschitz coefficients, but the equation is extremely complex, making the study of their well-posedness very challenging. 
Inspired by the Hu et al. \cite{HJX20}, we utilize the relationship between $(\hat{Y}(\cdot),\hat{Z}(\cdot))$ and $\hat{X}(\cdot)$ (see equation \eqref{relationship between X,Y and Z}) and estimate the remainder terms of the forward and backward equations simultaneously. 
Then we also obtain a class of linear BSDEs with unbounded random Lipschitz coefficients (see equation \eqref{Stochastic Lipschitz BSDE1} and \eqref{Stochastic Lipschitz BSDE2}), for which corresponding well-posedness results are available. 
Then, we obtain the required estimation results, thereby enabling the establishment of relationship \eqref{relation of adjoint and second-order super and sub-jets in spatial variable}. 
The proof idea for relationship \eqref{relationship of adjoint and super- and sub-jets in time variable} is similar to this.
\end{Remark}

Now, let us combine Theorem 4.1 and 4.2 to get the following result.

\begin{mythm}\label{relation of MP and DPP in time variable and spatial variable}
Under the condition of Theorems 4.1 and 4.2, for \text{a.e.} $s \in [t,T], \mathbf{P}$\text{-a.s.}, 
\begin{equation}\label{relationship of adjoint and super- and sub-jets in time variable and spatial variable}
\begin{aligned}
	[-\mathcal{H}_{1}(s,\bar{X}^{t,x;\bar{u}}(s),\bar{Z}^{t,x;\bar{u}}(s),\bar{u}(s)),\infty) \times \{p(s)\} \times [P(s),\infty)\subseteq D_{t+,x}^{1,2,+} V(s,\bar{X}^{t,x;\bar{u}}(s)), \\
	D_{t+,x}^{1,2,-} V(s,\bar{X}^{t,x;\bar{u}}(s)) \subseteq (-\infty,-\mathcal{H}_{1}(s,\bar{X}^{t,x;\bar{u}}(s),\bar{Z}^{t,x;\bar{u}}(s),\bar{u}(s))] \times \{p(s)\} \times (-\infty,P(s)].
\end{aligned}
\end{equation}
\end{mythm}

\begin{proof}
It is clear that \eqref{relationship of adjoint and super- and sub-jets in time variable and spatial variable} can be proved by combining the proofs of Theorems 4.1 and 4.2.
\end{proof}

\begin{Remark}
(1) When the risk-sensitive parameter $\mu \to 0$, the results of current paper reduce to the risk-neutral case in \cite{Z90-2,Z91} (see also \cite{YZ99}).

(2) When the value function $V$ is smooth, its second-order super- and sub-jets in the state variable $x$ reduce to:
\begin{equation}
\begin{aligned}
	D_x^{2,+}V(s,\bar{X}^{t,x;\bar{u}}(s))=\{V_x(s,\bar{X}^{t,x;\bar{u}}(s))\} \times [V_{xx}(s,\bar{X}^{t,x;\bar{u}}(s)),\infty) \\
	D_x^{2,-} V(s,\bar{X}^{t,x;\bar{u}}(s))= \{V_x(s,\bar{X}^{t,x;\bar{u}}(s))\} \times (-\infty, V_{xx}(s,\bar{X}^{t,x;\bar{u}}(s))].
\end{aligned}
\end{equation}
Consequently, the inclusions \eqref{relation of adjoint and second-order super and sub-jets in spatial variable} reduce to:
\begin{equation}
	p(s) = V_x(s,\bar{X}^{t,x;\bar{u}}(s)), \qquad V_{xx}(s,\bar{X}^{t,x;\bar{u}}(s))\leq P(s) ,
\end{equation}
which coincide exactly with Item 1 of (17) and (18) in Theorem 1 in \cite{DS26}.

Moreover, the first-order super- and sub-jet in the time variable $t$ reduce to:
\begin{equation}
\begin{aligned}
	D_{t+}^{1,+}V(s,\bar{X}^{t,x;\bar{u}}(s)) =[V_{t}(s,\bar{X}^{t,x;\bar{u}}(s)),\infty), \\
	D_{t+}^{1,-}V(s,\bar{X}^{t,x;\bar{u}}(s))=(-\infty,V_{t}(s,\bar{X}^{t,x;\bar{u}}(s))].
\end{aligned}
\end{equation}
Consequently, the inclusions reduce to a single inequality:
\begin{equation}
\begin{aligned}
	V_{t}(s,\bar{X}^{t,x;\bar{u}}(s))&\leq -\mathcal{H}_{1}(s,\bar{X}^{t,x;\bar{u}}(s),\bar{Z}^{t,x;\bar{u}}(s),\bar{u}(s)).
\end{aligned}
\end{equation}
\end{Remark}

\section{Some examples}

In this section, we give two examples to illustrate the main results (Theorems \ref{relation of MP and DPP in spatial variable} and \ref{relation of MP and DPP in time variable}) of this paper.

$\mathbf{Example \ 5.1.}$ Consider the following controlled FBSDE $(n=1)$ for $s \in [t,T]$:
\begin{equation}\label{state FBSDE example 1}
\begin{cases}
\begin{aligned}
	&dX^{t, x; u}(s)=u(s)dW(s),\\
	&dY^{t, x; u}(s) = -\left[u(s)^2 + \frac{\mu}{2}|Z^{t,x;u}(s)|^2\right]ds + Z^{t,x;u}(s)dW(s),\\
	&X^{t, x; u}(t)=x,\ Y^{t, x; u}(T) = \arctan(X^{t, x; u}(T)),
\end{aligned}
\end{cases}
\end{equation}
where the risk-sensitive parameter satisfies $\mu \geq 1$. The control domain is $U=\{0, 1\}$ and the cost functional is defined by $Y^{t, x; u}(t)$. The corresponding generalized HJB equation becomes
\begin{equation}\label{HJB equation of example}
\begin{cases}
	v_t(t,x) + \inf\limits_{u \in U} \left\{u^2+\frac{\mu}{2}u^2v_x(t,x)^2+\frac{1}{2}u^2v_{xx}(t,x)\right\}=0,\\
	v(T,x)=\arctan{x}.
\end{cases}
\end{equation}
Next we verify that the solution to \eqref{HJB equation of example} is given by $V(t,x)=\arctan{x}$. Set $V(t,x)=a(t)\arctan{x}$, where $a(t)$ satisfies
\begin{equation*}
\begin{cases}
	\dot{a}(t)\arctan{x} + \inf\limits_{u \in U} u^2\left\{1+\frac{\mu}{2}\frac{a(t)^2}{(1+x^2)^2}-\frac{1}{2}\frac{2xa(t)}{(1+x^2)^2}\right\}=0,\\
	a(T)=1,
\end{cases}
\end{equation*}
that is 
\begin{equation*}
\begin{cases}
	\dot{a}(t)\arctan{x} + \inf\limits_{u \in U} u^2\left\{\frac{2(1+x^4)+3x^2+(\mu-1)a(t)^2+(x-a(t))^2}{2(1+x^2)^2}\right\}=0,\\
	a(T)=1.
\end{cases}
\end{equation*}
Noting that $\frac{2(1+x^4)+3x^2+(\mu-1)a(t)^2+(x-a(t))^2}{2(1+x^2)^2}>0$, then $\bar{u}(\cdot)=0$ which leads to $a(t)=1$.  Thus  $V(t,x)=\arctan{x}$.  

Moreover, the first-order and second-order adjoint equations become
\begin{equation}\label{the first-order adjoint equation of example}
\left\{
\begin{aligned}
	dp(s) &= -\mu \bar{Z}^{t,x;\bar{u}}(s)q(s)ds + q(s)dW(s),\\
	p(T) &=\frac{1}{1+\bar{X}^{t,x;\bar{u}}(T)^2},
\end{aligned}
\right.
\end{equation}
and 
\begin{equation}\label{the second-order adjoint equation of example}
\left\{
\begin{aligned}
	dP(s) &= -\big[\mu q(s)^{2} + \mu Q(s)\bar{Z}^{t,x;\bar{u}}(s)\big]ds + Q(s)dW(s),\\
	P(T) &= -\frac{2\bar{X}^{t,x;\bar{u}}(T)}{(1+\bar{X}^{t,x;\bar{u}}(T)^2)^2}.
\end{aligned}
\right.
\end{equation}
For $\bar{u}(\cdot)=0$, it is easy to verify that $\bar{X}^{t, x; \bar{u}}(s)=x,\bar{Y}^{t, x; \bar{u}}(s)=\arctan x,\bar{Z}^{t, x; \bar{u}}(s)=0$ is the optimal trajectory. 
Then $V(t,x)=\bar{Y}^{t, x; \bar{u}}(t)=\arctan{x}$ is the value function. 
The solutions of the first-order and second-order adjoint equations for the optimal control $\bar{u}(\cdot)=0$ are $(p(s),q(s))=\left(\frac{1}{1+x^2},0\right)$, $(P(s),Q(s))=\left(-\frac{2x}{(1+x^2)^2},0\right)$. 
In this case, $\mathcal{H}_{1}(s,\bar{X}^{t,x;\bar{u}}(s),\bar{Z}^{t,x;\bar{u}}(s),\bar{u}(s))=0$. 
By \eqref{second-order super- and sub-jets in the spatial variable} and \eqref{right super- and sub-jets in the time variable}, we have for $s \in [t,T]$,
\begin{equation}\label{relation of MP and DPP example 1}
\begin{aligned}
	&D_x^{2,+} V(s,\bar{X}^{t, x; \bar{u}}(s))=\left\{\frac{1}{1+x^2}\right\} \times \left[-\frac{2x}{(1+x^2)^2}, \infty\right),\\
	&D_x^{2,-} V(s,\bar{X}^{t, x; \bar{u}}(s))=\left\{\frac{1}{1+x^2}\right\} \times \left(-\infty, -\frac{2x}{(1+x^2)^2}\right],\\
	&D_{t^+}^{1,+} V(s,\bar{X}^{t,x;\bar{u}}(s))= [0,\infty),\quad D_{t^+}^{1,-} V(s,\bar{X}^{t,x;\bar{u}}(s))=(-\infty, 0].
\end{aligned}
\end{equation}
Thus, the relations \eqref{relation of adjoint and second-order super and sub-jets in spatial variable} and \eqref{relationship of adjoint and super- and sub-jets in time variable} hold.

$\mathbf{Example \ 5.2.}$ Consider the following controlled FBSDE $(n=1)$ for $s \in [t,T]$:
\begin{equation}\label{state FBSDE example 2}
\begin{cases}
\begin{aligned}
	&dX^{t, x; u}(s)=-cu(s)X^{t, x; u}(s)ds+u(s)X^{t, x; u}(s)dW(s),\\
	&dY^{t, x; u}(s) = -\bigg[u(s)\left(\frac{cX^{t, x; u}(s)}{1+(X^{t, x; u}(s))^2}+\frac{(X^{t, x; u}(s))^3}{(1+(X^{t, x; u}(s))^2)^2}-\frac{\mu (X^{t, x; u}(s))^2}{2(1+(X^{t, x; u}(s))^2)^2}\right)\\
	&\qquad\qquad\qquad-\lambda(1-u(s))\arctan(X^{t, x; u}(s))+ \frac{\mu}{2}(Z^{t,x;u}(s))^2\bigg]ds +  Z^{t,x;u}(s)dW(s),\\
	&X^{t, x; u}(t)=x,\ Y^{t, x; u}(T) = \arctan(X^{t, x; u}(T)),
\end{aligned}
\end{cases}
\end{equation}
where $\mu>0$ is the risk-sensitive parameter, and the constants $c$ and $\lambda$ satisfy $0 < \mu < \min\{\mu^*, c\}$ and $0<\lambda T<1$. The control domain is $U=\{0,1\}$, and the cost functional is given by $Y^{t,x;u}(t)$.  The corresponding generalized HJB equation becomes
\begin{equation}\label{HJB equation of example2}
\begin{cases}
	\begin{aligned}
	&v_t(t,x) + \inf\limits_{u \in U} \bigg\{-\lambda(1-u)\arctan{x}+u\left(\frac{cx}{1+x^2}+\frac{x^3}{(1+x^2)^2}-\frac{\mu x^2}{2(1+x^2)^2}\right)\\
	&\quad
	-cuxv_{x}(t,x)+\frac{\mu}{2}u^2x^2v_x(t,x)^2+\frac{1}{2}u^2x^2v_{xx}(t,x)\bigg\}=0,\\
	&v(T,x)=\arctan{x}.
	\end{aligned}
\end{cases}
\end{equation} 
Set $V(t,x)=a(t)\arctan x$, where $0<a(t)\le 1$. Substituting into \eqref{HJB equation of example2} yields
\begin{equation*}
\begin{cases}
\begin{aligned}
	&\dot{a}(t)\arctan{x} + \inf\limits_{u \in U} \bigg\{-\lambda(1-u)\arctan{x}+u\left(\frac{cx}{1+x^2}+\frac{x^3}{(1+x^2)^2}-\frac{\mu x^2}{2(1+x^2)^2}\right)\\
	&\quad
	-cux\frac{a(t)}{1+x^2}+\frac{\mu}{2}u^2x^2\frac{a(t)^2}{(1+x^2)^2}-\frac{x^3a(t)u^{2}}{(1+x^2)^2}\bigg\}=0,\\
	&a(T)=1,
\end{aligned}
\end{cases}
\end{equation*} 
which simplifies to 
\begin{equation}\label{ODE of a(t)}
\begin{cases}
\begin{aligned}
	&\dot{a}(t)\arctan{x} + \inf \bigg\{-\lambda\arctan{x},\frac{(1-a(t))x}{(1+x^2)^2} \left( c(1+x^2) + x^2 - \frac{\mu(1+a(t))x}{2} \right)\bigg\}=0,\\
	&a(T)=1.
\end{aligned}
\end{cases}
\end{equation}

If $x<0$, we have $-\lambda\arctan{x}>0$, $\frac{(1-a(t))x}{(1+x^2)^2}\leq0$ and
\begin{equation*}
	\begin{aligned}
		& c(1+x^2) + x^2 - \frac{\mu(1+a(t))x}{2}
		\geq c(1 + x^2) + x^2  > 0,
	\end{aligned}
\end{equation*}
because $0<a(t)\leq1$. So the infimum is $\frac{(1-a(t))x}{(1+x^2)^2} \left( c(1+x^2) + x^2 - \frac{\mu(1+a(t))x}{2} \right)$ and $\bar u(s)=1$ is optimal. Substituting this into \eqref{ODE of a(t)} gives
\begin{equation*}
\begin{cases}
\begin{aligned}
	&\dot{a}(t)\arctan{x} + \frac{(1-a(t))x}{(1+x^2)^2} \left( c(1+x^2) + x^2 - \frac{\mu(1+a(t))x}{2} \right)=0,\\
	&a(T)=1,
\end{aligned}
\end{cases}
\end{equation*}
so $a(t)=1.$

If $x\geq0$, we have $-\lambda\arctan{x}\leq0$, $\frac{(1-a(t))x}{(1+x^2)^2}\geq0$ and
\begin{equation*}
\begin{aligned}
    & c(1+x^2) + x^2 - \frac{\mu(1+a(t))x}{2}
	\geq c(1 + x^2) + x^2 - \mu x \\
	&\quad\geq c(1 + x^2) + x^2 - cx 
	= c(x^2 - x + 1) + x^2 > 0,
\end{aligned}
\end{equation*}
because $c>\mu$ and $0<a(t)\leq1$. Hence the infimum is $-\lambda\arctan x$, and $\bar u(s)=0$ is an optimal control. Substituting this into \eqref{ODE of a(t)} gives
\begin{equation*}
\begin{cases}
\begin{aligned}
	&\dot{a}(t)\arctan{x} -\lambda\arctan{x}=0,\\
	&a(T)=1,
\end{aligned}
\end{cases}
\end{equation*}
so $a(t)=1-\lambda(T-t)$, which satisfies $0<a(t)<1$ due to $0<\lambda T<1$.

Thus, a viscosity solution of \eqref{HJB equation of example2} is given by
\begin{equation}\label{value function example2}
V(t, x) =
\begin{cases}
	\arctan{x},& \text{if } x < 0,\\
	(1-\lambda(T-t))\arctan{x},& \text{if } x \geq 0.
\end{cases}
\end{equation}
By the uniqueness result for the viscosity solution of the HJB equation, for $0<\mu<\mu^*$, $V(t,x)$ coincides with the value function of our problem. Moreover, the first-order and the second-order adjoint equations become
\begin{equation}\label{the first-order adjoint equation of example2}
\left\{
\begin{aligned}
	dp(s) &= -\bigg[ c\bar{u}(s) \frac{1 - (\bar{X}^{t,x;\bar{u}}(s))^2}{(1 + (\bar{X}^{t,x;\bar{u}}(s))^2)^2} +\bar{u}(s) \frac{(\bar{X}^{t,x;\bar{u}}(s))^2(3 - (\bar{X}^{t,x;\bar{u}}(s))^2)}{(1 + (\bar{X}^{t,x;\bar{u}}(s))^2)^3} \\ 
	&\qquad+\mu\bar{u}(s)\frac{\bar{X}^{t,x;\bar{u}}(s)((\bar{X}^{t,x;\bar{u}}(s))^2 - 1)}{(1 + (\bar{X}^{t,x;\bar{u}}(s))^2)^3} 
	- \frac{\lambda(1 - \bar{u}(s))}{1 + (\bar{X}^{t,x;\bar{u}}(s))^2}-c\bar{u}(s)p(s)\\
	&\qquad+\bar{u}(s)q(s) +\mu\bar{Z}^{t,x;\bar{u}}(s)(\bar{u}(s)p(s)+q(s))\bigg]ds + q(s)dW(s),\\
	p(T) &=\frac{1}{1+\bar{X}^{t,x;\bar{u}}(T)^2},
\end{aligned}
\right.
\end{equation}
and
\begin{equation}\label{the second-order adjoint equation of example2}
\left\{
\begin{aligned}
	dP(s)&= -\bigg[-2c\bar{u}(s)P(s)+(\bar{u}(s))^2(P(s)+\mu p(s)^2)+2\bar{u}(s)(Q(s)+\mu P(s)\bar{Z}^{t,x;\bar{u}}(s)+\mu p(s)q(s))\\
	&\quad+ \frac{2\lambda(1 - \bar{u}(s))\bar{X}^{t,x;\bar{u}}(s)}{(1 + (\bar{X}^{t,x;\bar{u}}(s))^2)^2} +\bar{u}(s)\frac{2\bar{X}^{t,x;\bar{u}}(s)(3 - 8(\bar{X}^{t,x;\bar{u}}(s))^2 + (\bar{X}^{t,x;\bar{u}}(s))^4)}{(1 + (\bar{X}^{t,x;\bar{u}}(s))^2)^4}\\  
	&\quad- \mu \bar{u}(s)\frac{1 - 8(\bar{X}^{t,x;\bar{u}}(s))^2 + 3(\bar{X}^{t,x;\bar{u}}(s))^4}{(1 + (\bar{X}^{t,x;\bar{u}}(s))^2)^4} 
	+\bar{u}(s) \frac{2c \, \bar{X}^{t,x;\bar{u}}(s)((\bar{X}^{t,x;\bar{u}}(s))^2 - 3)}{(1 + (\bar{X}^{t,x;\bar{u}}(s))^2)^3}\\
	&\quad+\mu q(s)^{2} +\mu Q(s)\bar{Z}^{t,x;\bar{u}}(s)\bigg]ds + Q(s)dW(s),\\
	P(T)&= -\frac{2\bar{X}^{t,x;\bar{u}}(T)}{(1+(\bar{X}^{t,x;\bar{u}}(T))^2)^2}.
\end{aligned}
\right.
\end{equation}
Now let us consider $t=0,x=0$. It is clear that $\bar{u}(\cdot)=0$ is an optimal control.
It is easy to verify that $\bar{X}^{0, 0; \bar{u}}(s)=0,\bar{Y}^{0, 0; \bar{u}}(s)=0,\bar{Z}^{0, 0; \bar{u}}(s)=0$ is the corresponding optimal trajectory, which implies $\mathcal{H}_{1}(s,\bar{X}^{0,0;\bar{u}}(s),\bar{Z}^{0,0;\bar{u}}(s),\bar{u}(s))=0$. 
The solutions of the first-order and second-order adjoint equations for the optimal control $\bar{u}(\cdot)=0$ are $(p(s),q(s))=(1-\lambda(T-s),0), (P(s),Q(s))=(0,0)$. 
In this case, $V_{x}(t,x)$ does not exist along the whole state $\bar{X}^{0, 0;\bar{u}}(s), s \in [t,T)$. 
However, by \eqref{second-order super- and sub-jets in the spatial variable} and \eqref{right super- and sub-jets in the time variable}, we have for $s \in [t,T)$,
\begin{equation}\label{relation of MP and DPP example 2}
\begin{aligned}
	&D_x^{2,+} V(s,\bar{X}^{0, 0; \bar{u}}(s))=(1-\lambda(T-s), 1) \times \mathbb{R} \cup \{1-\lambda(T-s), 1\} \times [0, \infty),\\
	&D_x^{2,-} V(s,\bar{X}^{0, 0; \bar{u}}(s))=\emptyset,\\
	&D_{t^+}^{1,+} V(s,\bar{X}^{0, 0; \bar{u}}(s))= [0,\infty),\
	D_{t^+}^{1,-} V(s,\bar{X}^{0, 0; \bar{u}}(s))=(-\infty, 0].
\end{aligned}
\end{equation}
Thus, the relations \eqref{relation of adjoint and second-order super and sub-jets in spatial variable} and \eqref{relationship of adjoint and super- and sub-jets in time variable} hold.

\section{Concluding remarks}

In this paper, we have established a nonsmooth version of the relationship between the general maximum principle and dynamic programming principle for the risk-sensitive stochastic optimal control problems, where the control domain is not necessarily convex. 
By introducing a BSDE with quadratic generator, the original problem is equivalent to a stochastic recursive optimal control problem of a forward-backward system with quadratic generators. 
We extend the work of Zhou \cite{Z90-2,Z91} to the risk-sensitive optimal control problem and partially generalize to the work by Nie et al. \cite{NSW17} to a forward-backward system, where the backward equation has quadratic growth. 
With the help of viscosity solutions, we give the relationship among the adjoint process, generalized Hamiltonian function and value function when the value function being not necessarily smooth. 
At last, we have given two examples to illustrate the theoretical results.

It would be interesting to investigate the stochastic verification theorem of risk-sensitive stochastic optimal control problems, and the relationship between MP and DPP for the risk-sensitive stochastic optimal control problem of jump diffusions. 
We will research these topics in the near future.

\section*{Acknowledgements}

The authors would like to thank two anonymous reviewers for their valuable comments on the paper which led to significant improvement on the presentation and quality of the paper.

\section*{Appendix}
\setcounter{section}{0}

\subsection{Proof of Theorem \ref{relation of MP and DPP in spatial variable}}

\begin{proof}
We split the proof into several steps. In the following, the constant $C>0$ will change from line to line for simplicity. 

Step 1. Variational equations.
  
Fix an $s \in [t,T]$. For any $x' \in \mathbb{R}^n$, denote by $(X^{s,x';\bar{u}}(\cdot),Y^{s,x';\bar{u}}(\cdot),Z^{s,x';\bar{u}}(\cdot))$ the solution to the following FBSDE:
\begin{equation}\label{recursive FBSDE1}
\begin{cases}
\begin{aligned}
	&dX^{s, x'; \bar{u}}(r)=b(r,X^{s, x'; \bar{u}}(r),\bar{u}(r))dr +\sigma(r,X^{s, x'; \bar{u}}(r),\bar{u}(r))dW(r),\\
	&dY^{s, x'; \bar{u}}(r) = -\left[f(r, X^{s, x'; \bar{u}}(r), \bar{u}(r)) + \frac{\mu}{2}|Z^{s, x'; \bar{u}}(r)|^2\right]dr + Z^{s, x'; \bar{u}}(r)dW(r),\\
	&X^{s, x'; \bar{u}}(s)=x',\ Y^{s, x'; \bar{u}}(T) = h(X^{s, x'; \bar{u}}(T)),\ r \in [s,T].
\end{aligned}
\end{cases}
\end{equation}
It is clear that \eqref{recursive FBSDE1} can be regarded as a FBSDE on $(\Omega,\mathcal{F},\{\mathcal{F}_r^t\}_{r \geq t},\mathbf{P}(\cdot|\mathcal{F}_s^t)(\omega))$ for $\mathbf{P}\text{-a.s.}\, \omega$, 
where $\mathbf{P}(\cdot|\mathcal{F}_s^t)(\omega)$ is the regular conditional probability given $\mathcal{F}_s^t$ defined on $(\Omega,\mathcal{F})$. 
For any $s \leq r \leq T$, set $\hat{X}(r) := X^{s,x';\bar{u}}(r) - \bar{X}^{t,x;\bar{u}}(r)$, $\hat{Y}(r) := Y^{s,x';\bar{u}}(r) - \bar{Y}^{t,x;\bar{u}}(r)$, $\hat{Z}(r) := Z^{s,x';\bar{u}}(r) - \bar{Z}^{t,x;\bar{u}}(r)$. 
Thus by Lemmas \ref{solvability of recursive FBSDE} and \ref{priori estimate of BSDE}, we have for any $p \geq 2$,\vspace{-3pt}
\begin{equation}\label{estimate of recursive FBSDE1}
	\mathbb{E} \left[\sup\limits_{r \in [s,T]}\left(|\hat{X}(r)|^{p}+|\hat{Y}(r)|^{p}\right) + \Bigl( \int_s^T \bigl| \hat{Z}(r) \bigr|^2 dr \Bigr)^{\frac{p}{2}}\Big|\mathcal{F}_s^t  \right] \leq C|x'-\bar{X}^{t,x;\bar{u}}(s)|^{p},\quad \mathbf{P}\text{-a.s.}.
\end{equation}
Now we write the equation for  $(\hat{X}(\cdot),\hat{Y}(\cdot),\hat{Z}(\cdot))$ as the following FBSDE:
\begin{equation}\label{Variational equations 1}
\begin{cases}
\begin{aligned}
	&d\hat{X}(r)=\left[\bar{b}_x(r)\hat{X}(r)+\varepsilon_{1}(r)\right]dr +\left[\bar{\sigma}_x(r)\hat{X}(r)+\varepsilon_{2}(r)\right]dW(r),\\
	&d\hat{Y}(r) = -\left[\bar{f}_x(r)\hat{X}(r) + \mu \bar{Z}^{t, x; \bar{u}}(r)\hat{Z}(r) + \varepsilon_{3}(r)\right]dr + \hat{Z}(r)dW(r),\\
	&\hat{X}(s)=x'-\bar{X}^{t, x; \bar{u}}(s),\ \hat{Y}(T)=h_x(\bar{X}^{t,x;\bar{u}}(T))\hat{X}(T)+\varepsilon_{4}(T),\ r \in [s,T],
\end{aligned}
\end{cases}
\end{equation}
where
\begin{equation*}
\begin{aligned}
	\varepsilon_{1}(r) &:= \int_{0}^{1}[b_x(r,\bar{X}^{t,x;\bar{u}}(r) + \theta\hat{X}(r),\bar{u}(r))-\bar{b}_x(r)]\hat{X}(r)d\theta, \\
	\varepsilon_{2}(r) &:= \int_{0}^{1}[\sigma_x(r,\bar{X}^{t,x;\bar{u}}(r) + \theta\hat{X}(r),\bar{u}(r))-\bar{\sigma}_x(r)]\hat{X}(r)d\theta,\\
	\varepsilon_{3}(r) &:= \int_{0}^{1}[f_x(r,\bar{X}^{t,x;\bar{u}}(r) + \theta\hat{X}(r),\bar{u}(r))-\bar{f}_x(r)]\hat{X}(r)d\theta + \frac{\mu}{2}|\hat{Z}(r)|^2,\\
	\varepsilon_{4}(T) &:= \int_{0}^{1}[h_x(\bar{X}^{t,x;\bar{u}}(T) + \theta\hat{X}(T))-h_x(\bar{X}^{t,x;\bar{u}}(T))]\hat{X}(T)d\theta.
\end{aligned}
\end{equation*}

Step 2. Estimates of the remainder terms of FBSDE.

For any $p \geq 2$, we have 
\begin{equation}\label{eatimates of remainder terms of FBSDE}
\begin{aligned}
	\mathbb{E} \left[\left(\int_{s}^{T}|\varepsilon_{i}(r)|dr\right)^{p} \Big|\mathcal{F}_s^t  \right] &\leq C|x'-\bar{X}^{t,x;\bar{u}}(s)|^{2p},\ i=1,2,3, \ \mathbf{P}\text{-a.s.},\\
	\mathbb{E} \left[\int_{s}^{T}|\varepsilon_{i}(r)|^{p}dr \Big|\mathcal{F}_s^t  \right] &\leq C|x'-\bar{X}^{t,x;\bar{u}}(s)|^{2p},\ i=1,2, \ \mathbf{P}\text{-a.s.},\\
	\mathbb{E} \left[|\varepsilon_{4}(T)|^{p} \big|\mathcal{F}_s^t  \right] &\leq C|x'-\bar{X}^{t,x;\bar{u}}(s)|^{2p},\ \mathbf{P}\text{-a.s.}.
\end{aligned}
\end{equation}
Here the inequality $\mathbb{E} \left[\int_{s}^{T}|\varepsilon_{i}(r)|^{p}dr |\mathcal{F}_s^t \right] \leq C|x'-\bar{X}^{t,x;\bar{u}}(s)|^{2p},\ \mathbf{P}\text{-a.s.}$, 
means that for $\mathbf{P}\text{-a.s.}\ \omega$ fixed, $\mathbb{E} \left[\int_{s}^{T}|\varepsilon_{i}(r)|^{p}dr |\mathcal{F}_s^t \right](\omega) \leq C(|x'-\bar{X}^{t,x;\bar{u}}(s,w)|^{2p}$. 
Such notation has a similar meaning for other estimates in \eqref{eatimates of remainder terms of FBSDE} as well as in what follows in the paper.

Now we prove \eqref{eatimates of remainder terms of FBSDE}. By Assumption 4, we have
\begin{equation*}
\begin{aligned}
	|\varepsilon_{1}(r)| &\leq \int_{0}^{1}|b_x(r,\bar{X}^{t,x;\bar{u}}(r) + \theta\hat{X}(r),\bar{u}(r))-\bar{b}_x(r)|d\theta |\hat{X}(r)|
     \leq C|\hat{X}(r)|^{2},\\
    |\varepsilon_{3}(r)| &\leq \int_{0}^{1}|f_x(r,\bar{X}^{t,x;\bar{u}}(r) + \theta\hat{X}(r),\bar{u}(r))-\bar{f}_x(r)|d\theta |\hat{X}(r)|+\frac{\mu}{2}|\hat{Z}(r)|^2
     \leq C(|\hat{X}(r)|^{2}+|\hat{Z}(r)|^2).
\end{aligned}
\end{equation*}
Similarly, we have $|\varepsilon_{2}(r)|\leq C|\hat{X}(r)|^{2}$ and $|\varepsilon_{4}(T)|\leq C|\hat{X}(T)|^{2}$.
Then, by \eqref{estimate of recursive FBSDE1}, we obtain that
\begin{equation*}
\begin{aligned}
	 \mathbb{E} \left[ \left( \int_{s}^{T} |\varepsilon_{i}(r)| \, dr \right)^{p} \, \big| \, \mathcal{F}_s^t \right] 
	&\leq C \mathbb{E} \left[ \sup_{r \in [s,T]} |\hat{X}(r)|^{2p} + \left( \int_s^T |\hat{Z}(r)|^2 dr \right)^{p} \Big| \mathcal{F}_s^t \right] \\
	&\leq C |x' - \bar{X}^{t,x;\bar{u}}(s)|^{2p}, \quad i = 1,2,3, \quad \mathbf{P}\text{-a.s.}, \\
	 \mathbb{E} \left[ \int_{s}^{T} |\varepsilon_{i}(r)|^{p} dr \Big| \mathcal{F}_s^t \right] 
	&\leq C \, \mathbb{E} \left[ \sup_{r \in [s,T]} |\hat{X}(r)|^{2p}  \Big| \mathcal{F}_s^t \right] \\
	&\leq C |x' - \bar{X}^{t,x;\bar{u}}(s)|^{2p}, \quad i = 1,2, \quad \mathbf{P}\text{-a.s.}, \\
	 \mathbb{E} \left[ |\varepsilon_{4}(T)|^{p} \, \big| \, \mathcal{F}_s^t \right] 
	&\leq C \mathbb{E} \left[ |\hat{X}(T)|^{2p} \, \big| \, \mathcal{F}_s^t \right] 
	 \leq C |x' - \bar{X}^{t,x;\bar{u}}(s)|^{2p}, \quad \mathbf{P}\text{-a.s.}.
\end{aligned}
\end{equation*}

Step 3. Relationship between $\hat{X}(\cdot)$ and $(\hat{Y}(\cdot),\hat{Z}(\cdot))$. 

Applying It\^o's formula to $\langle p(\cdot),\hat{X}(\cdot) \rangle$, we get
\begin{equation}\label{relationship between X,Y and Z}
\begin{aligned}
	 \hat{Y}(r)&=\langle p(r),\hat{X}(r) \rangle + \gamma(r),\\
	 \hat{Z}(r)&= \langle q(r) + \bar{\sigma}_{x}(r) p(r), \hat{X}(r) \rangle + \langle p(r), \varepsilon_2(r) \rangle + \kappa(r),
\end{aligned}
\end{equation}
where $p(\cdot)$ is the solution to first order adjoint equation \eqref{the first-order adjoint equation}, and $(\gamma(r),\kappa(r))$ is the solution to the following stochastic Lipschitz BSDE:
\begin{equation}\label{Stochastic Lipschitz BSDE1}
\left\{
\begin{aligned}
	d\gamma(r) = &-\Bigl[ \mu \bar{Z}^{t,x;\bar{u}}(r)\kappa(r) + p(r)^{\top}\varepsilon_{1}(r) + q(r)^{\top}\varepsilon_{2}(r) + \varepsilon_{3}(r) \\
	&\quad + \mu \bar{Z}^{t,x;\bar{u}}(r)p(r)^{\top}\varepsilon_{2}(r) \Bigr] dr + \kappa(r) dW(r), \\
	\gamma(T) = &\, \varepsilon_{4}(T).
\end{aligned}
\right.
\end{equation}

Step 4. Variation of the stochastic Lipschitz BSDE.

Note $\bar{Z}^{t,x; \bar{u}}\cdot W \in BMO$, then by the estimate of stochastic Lipschitz BSDEs (see Lemma \ref{solvability of Stochastic Lipschitz BSDEs}) for \eqref{Stochastic Lipschitz BSDE1} and by \eqref{eatimates of remainder terms of FBSDE}, we obtain that, for each $p\geq2$,
\begin{equation}\label{the estimate of stochastic Lipschitz BSDE1}
\begin{aligned}
	&\mathbb{E} \left[\sup\limits_{r \in [s,T]}|\gamma(r)|^{p} + \Bigl( \int_s^T \bigl| \kappa(r) \bigr|^2 dr \Bigr)^{\frac{p}{2}}\Big|\mathcal{F}_s^t  \right]\\
	&\leq C\left(\mathbb{E} \left[|\varepsilon_4(T)|^{p\bar{q}^2}+\Bigl( \int_s^T | p(r)^{\top}\varepsilon_{1}(r) + q(r)^{\top}\varepsilon_{2}(r) +\varepsilon_{3}(r)\right.\right.\\
    &\qquad\quad +\left.\left.\mu\bar{Z}^{t,x;\bar{u}}(r)p(r)^{\top}\varepsilon_{2}(r)| dr \Bigr)^{{p\bar{q}}^{2}}\Big|\mathcal{F}_s^t \right]\right)^{\frac{1}{\bar{q}^2}}\\
	&\leq C\biggl\{\left(\mathbb{E} \left[|\varepsilon_4(T)|^{p\bar{q}^2} \big|\mathcal{F}_s^t \right]\right)^{\frac{1}{\bar{q}^2}} 
     +\left(\mathbb{E} \left[\left(\int_s^T|\varepsilon_{1}(r)|dr\right)^{p\bar{q}^{2}}\Big|\mathcal{F}_s^t \right]\right)^{\frac{1}{\bar{q}^{2}}}\\
	&\qquad+\left(\mathbb{E} \left[\left(\int_s^T  |q(r)|^{2}dr\right)^{p\bar{q}^{2}}\big|\mathcal{F}_s^t \right]\right)^{\frac{1}{2\bar{q}^{2}}}
     \cdot \left(\mathbb{E} \left[\int_s^T  |\varepsilon_{2}(r)|^{2p\bar{q}^{2}}dr\Big|\mathcal{F}_s^t \right]\right)^{\frac{1}{2\bar{q}^{2}}}\\
	&\qquad+\left(\mathbb{E} \left[\left(\int_s^T  |\bar{Z}^{t,x;\bar{u}}(r)|^{2}dr\right)^{p\bar{q}^{2}}\big|\mathcal{F}_s^t \right]\right)^{\frac{1}{2\bar{q}^{2}}} 
     \cdot \left(\mathbb{E} \left[\int_s^T  |\varepsilon_{2}(r)|^{2p\bar{q}^{2}}dr\Big|\mathcal{F}_s^t \right]\right)^{\frac{1}{2\bar{q}^{2}}}\\
	&\qquad+\left(\mathbb{E} \left[\left(\int_s^T  |\varepsilon_{3}(r)|dr\right)^{p\bar{q}^{2}}\Big|\mathcal{F}_s^t \right]\right)^{\frac{1}{\bar{q}^{2}}}\biggr\}\\
	&\leq C|x'-\bar{X}^{t,x;\bar{u}}(s)|^{2p}, \quad \mathbf{P}\text{-a.s.}.
\end{aligned}
\end{equation}
In the followings, we want to prove
\begin{equation}\label{the estimate of stochastic Lipschitz BSDE2}
	|\gamma(s)-\frac{1}{2}\hat{X}(s)^{\top}P(s)\hat{X}(s)|\leq\delta(|x'-\bar{X}^{t,x;\bar{u}}(s)|^{2}), \quad \mathbf{P}\text{-a.s.}.
\end{equation}
where $\delta(\cdot): [0,\infty)\to [0,\infty)$ is a deterministic, continuous and increasing function, independent of $x' \in \mathbb{R}^n$, with $\frac{\delta(r)}{r}\to0$ as $r\to0$, here and below, $\delta$ may change from line to line.
Define
\begin{equation*}
\begin{aligned}
	\tilde{\gamma}(r)&:=\frac{1}{2}\hat{X}(r)^{\top}P(r)\hat{X}(r),\\
	\tilde{\kappa}(r)&:=\hat{X}(r)^{\top}P(r)(\bar{\sigma}_{x}(r)\hat{X}(r)+\varepsilon_{2}(r))+\frac{1}{2}\hat{X}(r)^{\top}Q(r)\hat{X}(r).
\end{aligned}
\end{equation*}
Applying It\^o's formula to $\frac{1}{2}\hat{X}(r)^{\top}P(r)\hat{X}(r)$, we obtain that $(\tilde{\gamma}(r), \tilde{\kappa}(r))$ satisfies the following BSDE:
\begin{equation}\label{tilde gamma}
\left\{
\begin{aligned}
 	d\tilde{\gamma}(r) = &\Biggl\{ \varepsilon_{1}(r)^{\top}P(r)\hat{X}(r) + \frac{1}{2}\varepsilon_{2}(r)^{\top}P(r)\varepsilon_{2}(r) + \varepsilon_{2}(r)^{\top}P(r)\bar{\sigma}_x(r)\hat{X}(r)\\
 	&-\frac{1}{2}\langle p(r),\hat{X}(r)^{\top}\bar{b}_{xx}(r)\hat{X}(r) \rangle-\frac{1}{2}\langle q(r),\hat{X}(r)^{\top}\bar{\sigma}_{xx}(r)\hat{X}(r) \rangle-\frac{1}{2}\hat{X}(r)^{\top}\bar{f}_{xx}(r)\hat{X}(r)\\
 	&-\frac{1}{2}\mu\langle\bar{Z}^{t,x;\bar{u}}(r)p(r),\hat{X}(r)^{\top}\bar{\sigma}_{xx}(r)\hat{X}(r) \rangle - \mu\hat{X}(r)^{\top}\bar{\sigma}_x(r)\bar{Z}^{t,x;\bar{u}}(r)P(r)\hat{X}(r)\\
 	&-\frac{1}{2}\mu\bar{Z}^{t,x;\bar{u}}(r)\hat{X}(r)^{\top}Q(r)\hat{X}(r)-\frac{1}{2}\mu\hat{X}(r)^{\top}(\bar{\sigma}_x(r)p(r) + q(r)\bigr)(\bar{\sigma}_x(r)p(r) \\
 	&+ q(r)\bigr)^{\top}\hat{X}(r)+ \varepsilon_{2}(r)^{\top}Q(r)\hat{X}(r) \Biggr\} dr + \tilde{\kappa}(r) dW(r), \\
    \tilde{\gamma}(T) = &\ \frac{1}{2}\hat{X}(T)^{\top} h_{xx}\bigl(\bar{X}^{t,x;\bar{u}}(T)\bigr) \hat{X}(T),
\end{aligned}
\right.
\end{equation}
where $\hat{X}(r)^{\top}\bar{\phi}_{xx}(r)\hat{X}(r)=(\hat{X}(r)^{\top}\bar{\phi}^{1}_{xx}(r)\hat{X}(r),\cdots,\hat{X}(r)^{\top}\bar{\phi}^{n}_{xx}(r)\hat{X}(r))^{\top}$ for $\phi=b,\sigma,f$.

Set
\begin{equation*}
	\hat{\gamma}(r)=\gamma(r)-\tilde{\gamma}(r),\ \hat{\kappa}(r)=\kappa(r)-\tilde{\kappa}(r).
\end{equation*}
Replace $\varepsilon_{1}(r)$ by $\frac{1}{2}\hat{X}(r)^{\top}\bar{b}_{xx}(r)\hat{X}(r)+\varepsilon_{5}(r)$, $\varepsilon_{2}(r)$ by $\frac{1}{2}\hat{X}(r)^{\top}\bar{\sigma}_{xx}(r)\hat{X}(r)+\varepsilon_{6}(r)$, $\varepsilon_{3}(r)$ by $\frac{1}{2}\hat{X}(r)^{\top}\bar{f}_{xx}(r)\hat{X}(r)+\frac{1}{2}\mu|\hat{Z}(r)|^2+\varepsilon_{7}(r)$ and $\varepsilon_{4}(T)$ by $\frac{1}{2}\hat{X}(T)^{\top}h_{xx}(\bar{X}^{t,x;\bar{u}}(T))\hat{X}(T)+\varepsilon_{8}(T)$ in \eqref{Stochastic Lipschitz BSDE1}, 
where
\begin{equation*}
\begin{aligned}
	\varepsilon_{5}(r)&:=\hat{X}(r)^\top\int_0^1 \int_0^1 \lambda \left[ b_{xx}\bigl(r, \bar{X}^{t,x; \bar{u}}(r) + \theta \lambda \hat{X}(r), \bar{u}(r)\bigr) - \bar{b}_{xx}(r) \right] d\lambda \, d\theta \; \hat{X}(r),\\
	\varepsilon_{6}(r)&:=\hat{X}(r)^\top\int_0^1 \int_0^1 \lambda \left[ \sigma_{xx}\bigl(r, \bar{X}^{t,x; \bar{u}}(r) + \theta \lambda \hat{X}(r), \bar{u}(r)\bigr) - \bar{\sigma}_{xx}(r) \right] d\lambda \, d\theta \; \hat{X}(r),\\
	\varepsilon_{7}(r)&:=\hat{X}(r)^\top\int_0^1 \int_0^1 \lambda \left[ f_{xx}\bigl(r, \bar{X}^{t,x; \bar{u}}(r) + \theta \lambda \hat{X}(r), \bar{u}(r)\bigr) - \bar{f}_{xx}(r) \right] d\lambda \, d\theta \; \hat{X}(r),\\
	\varepsilon_{8}(T)&:=\hat{X}(T)^\top\int_0^1 \int_0^1 \lambda \left[ h_{xx}\bigl(\bar{X}^{t,x; \bar{u}}(T) + \theta \lambda \hat{X}(T)\bigr) - h_{xx}(\bar{X}^{t,x; \bar{u}}(T) ) \right] d\lambda \, d\theta \; \hat{X}(T).\\
\end{aligned}
\end{equation*}
Then we can verify that $(\hat{\gamma}(r),\hat{\kappa}(r))$ satisfies the following stochastic Lipschitz BSDE:
\begin{equation}\label{Stochastic Lipschitz BSDE2}
\left\{
\begin{aligned}
	d\hat{\gamma}(r)&= -\left[\mu\bar{Z}^{t,x; \bar{u}}(r)\hat{\kappa}(r) + I(r)\right]dr+\hat{\kappa}(r)dW(r),\\
	\hat{\kappa}(T)&= \varepsilon_{8}(T).
\end{aligned}
\right.
\end{equation}
where
\begin{equation*}
\begin{aligned}
  I(r):= &\ \mu\bar{Z}^{t,x; \bar{u}}(r)\varepsilon_{2}(r)^{\top}P(r)\hat{X}(r) + \varepsilon_{5}(r)^{\top}p(r) +\varepsilon_{7}(r)+\varepsilon_{6}(r)^{\top}\left(q(r)+\mu\bar{Z}^{t,x; \bar{u}}(r)p(r)\right)\\
  &+\mu\left(p(r)^{\top}\varepsilon_{2}(r)+\kappa(r)\right)\left(\bar{\sigma}_x(r)p(r) + q(r)\right)^{\top}\hat{X}(r)+\frac{1}{2}\mu\left(p(r)^{\top}\varepsilon_{2}(r)+\kappa(r)\right)^2\\
  &+\varepsilon_{1}(r)^{\top}P(r)\hat{X}(r)+\frac{1}{2}\varepsilon_{2}(r)^{\top}P(r)\varepsilon_{2}(r)+\varepsilon_{2}(r)^{\top}P(r)\bar{\sigma}_x(r)\hat{X}(r)+\varepsilon_{2}(r)^{\top}Q(r)\hat{X}(r).
\end{aligned}
\end{equation*}
Note $\bar{Z}^{t,x; \bar{u}}\cdot W \in BMO$, then by the estimate of the stochastic Lipschitz BSDEs (see Lemma 2.4 ), we obtain that
\begin{equation*}
\begin{aligned}
	|\hat{\gamma}(s)|^2
	&\leq\mathbb{E} \left[\sup\limits_{r \in [s,T]}|\hat{\gamma}(r)|^{2}\big|\mathcal{F}_s^t  \right]
	 \leq C\left(\mathbb{E} \left[ |\varepsilon_{8}(T)|^{2\bar{q}^2}+\left(\int_s^T|I(r)|dr\right)^{2\bar{q}^2} \Big|\mathcal{F}_s^t  \right]\right)^{\frac{1}{\bar{q}^2}}\\
	&\leq C\left\{\left(\mathbb{E} \left[ |\varepsilon_{8}(T)|^{2\bar{q}^2}\big|\mathcal{F}_s^t \right]\right)^{\frac{1}{\bar{q}^2}}
     +\left(\mathbb{E} \left[ \left(\int_s^T|I(r)|dr\right)^{2\bar{q}^2}\Big|\mathcal{F}_s^t \right]\right)^{\frac{1}{\bar{q}^2}}\right\}.
\end{aligned}
\end{equation*}
Denote
\begin{equation*}
\begin{aligned}
	\rho(r) &:= \sum_{i=1}^{3} \int_{0}^{1} \int_{0}^{1} \lambda \left|\psi^{i}_{xx}(r, \bar{X}^{t,x; \bar{u}}(r) + \theta\lambda\hat{X}(r), \bar{u}(r)) - \bar{\psi}^{i}_{xx}(r) \right| d\lambda d\theta \\
	&\quad + \sum_{i=1}^{2} \int_{0}^{1} \left| \psi^{i}_{x}(r, \bar{X}^{t,x; \bar{u}}(r) + \theta\hat{X}(r), \bar{u}(r)) - \bar{\psi}^{i}_{x}(r) \right| d\theta,
\end{aligned}
\end{equation*}
where $\psi^1=b, \psi^2=\sigma, \psi^3=f.$ Thus, we have
\begin{equation}\label{estimate1}
	|\varepsilon_{i}(r)|\leq\rho(r)|\hat{X}(r)|,\quad |\varepsilon_{j}(r)|\leq\rho(r)|\hat{X}(r)|^2,\ i=1,2 ,\ j=5,6,7.
\end{equation}
Since $p(\cdot)$ and $\bar{\sigma}_x(\cdot)$ is bounded, by \eqref{estimate1}, we get
\begin{equation*}
\begin{aligned}
	|I(r)| \leq C\Big\{&(1+|P(r)|)\rho(r)|\bar{Z}^{t,x; \bar{u}}(r)||\hat{X}(r)|^2+(1+|q(r)|)\rho(r)|\hat{X}(r)|^2\\
	&+(1+|q(r)|)|\kappa(r)\hat{X}(r)|+\rho(r)|\kappa(r)\hat{X}(r)|+|\kappa(r)|^2\\
	&+(|P(r)|+|Q(r)|)\rho(r)|\hat{X}(r)|^2+(1+|P(r)|)\rho(r)^2|\hat{X}(r)|^2\Big\}.
\end{aligned}
\end{equation*}
Next, we estimate term by term. By the Lipschitz continuity of $h_{xx}$, we have
\begin{equation}
\begin{aligned}
	&\left(\mathbb{E} \left[ |\varepsilon_{8}(T)|^{2\bar{q}^2}\big|\mathcal{F}_s^t \right]\right)^{\frac{1}{\bar{q}^2}}
     \leq \left(\mathbb{E}\left[\left|\hat{X}(T)\right|^{8\bar{q}^2} \big| \mathcal{F}_s^t\right]\right)^{\frac{1}{2\bar{q}^2}} \\
	&\quad\cdot \left(\mathbb{E}\left[\left|\int_0^1\int_0^1 \lambda \left[h_{xx}(\bar{X}^{t,x;\bar{u}}(T)
     +\theta\lambda\hat{X}(T))-h_{xx}(\bar{X}^{t,x;\bar{u}}(T))\right] d\lambda d\theta \right|^{4\bar{q}^2} \Big| \mathcal{F}_s^t\right]\right)^{\frac{1}{2\bar{q}^2}} \\
	&\leq \delta\left(|x'-\bar{X}^{t,x;\bar{u}}(s)|^4\right).
\end{aligned}
\end{equation}
By H\"older's inequality, we get
\begin{equation}
\begin{aligned}
	&\left(\mathbb{E} \left[ \left(\int_s^T(1+|P(r)|)\rho(r)|\bar{Z}^{t,x; \bar{u}}(r)||\hat{X}(r)|^2dr\right)^{2\bar{q}^2}\Big|\mathcal{F}_s^t \right]\right)^{\frac{1}{\bar{q}^2}}\\
	&\leq C\left(\mathbb{E} \left[\left(1+\sup\limits_{r \in [s,T]}|P(r)|^{2\bar{q}^2}\right)\sup\limits_{r \in [s,T]}|\hat{X}(r)|^{4\bar{q}^2} \left(\int_s^T\rho(r)^2dr\right)^{\bar{q}^2}
     \left(\int_s^T|\bar{Z}^{t,x; \bar{u}}(r)|^2dr\right)^{\bar{q}^2}\Big|\mathcal{F}_s^t \right]\right)^{\frac{1}{\bar{q}^2}}\\
	&\leq C\left(\mathbb{E} \left[\left(1+\sup_{r \in [s,T]}|P(r)|^{2\bar{q}^2}\right)^2\Big|\mathcal{F}_s^t \right]\right)^{\frac{1}{2\bar{q}^2}}
	 \left(\mathbb{E} \left[ \sup_{r \in [s,T]}|\hat{X}(r)|^{16\bar{q}^2}\Big|\mathcal{F}_s^t \right]\right)^{\frac{1}{4\bar{q}^2}} \\
	&\quad \cdot \left(\mathbb{E} \left[ \left(\int_s^T \rho(r)^2 dr\right)^{8\bar{q}^2}\Big|\mathcal{F}_s^t \right]\right)^{\frac{1}{8\bar{q}^2}}
	\left(\mathbb{E} \left[ \left(\int_s^T |\bar{Z}^{t,x; \bar{u}}(r)|^2 dr\right)^{8\bar{q}^2}\Big|\mathcal{F}_s^t \right]\right)^{\frac{1}{8\bar{q}^2}}\\
	&\leq\delta\left(|x'-\bar{X}^{t,x;\bar{u}}(s)|^4\right).
\end{aligned}
\end{equation}
Similarly, we obtain
\begin{equation}\label{estimate2}
\begin{aligned}
	&\left(\mathbb{E} \left[ \left(\int_s^T(1+|q(r)|)\rho(r)|\hat{X}(r)|^2dr\right)^{2\bar{q}^2}\Big|\mathcal{F}_s^t \right]\right)^{\frac{1}{\bar{q}^2}}\\
	&\leq C\left(\mathbb{E} \left[\sup\limits_{r \in [s,T]}|\hat{X}(r)|^{4\bar{q}^2} \left(\int_s^T\rho(r)^2dr\right)^{\bar{q}^2}\left(\int_s^T(1+|q(r)|)^2dr\right)^{\bar{q}^2}\Big|\mathcal{F}_s^t \right]\right)^{\frac{1}{\bar{q}^2}}\\
	&\leq\delta\left(|x'-\bar{X}^{t,x;\bar{u}}(s)|^4\right),
\end{aligned}
\end{equation}
and the estimate for $\left(\mathbb{E} \left[ \left(\int_s^T(|P(r)|+|Q(r)|)\rho(r)|\hat{X}(r)|^2dr\right)^{2\bar{q}^2}\Big|\mathcal{F}_s^t \right]\right)^{\frac{1}{\bar{q}^2}}$ is similar to \eqref{estimate2}. 
By Holder's inequality, we get
\begin{equation}\label{estimate3}
\begin{aligned}
	&\left(\mathbb{E} \left[ \left(\int_s^T(1+|q(r)|)|\kappa(r)\hat{X}(r)|dr\right)^{2\bar{q}^2}\Big|\mathcal{F}_s^t \right]\right)^{\frac{1}{\bar{q}^2}}\\
	&\leq C\left(\mathbb{E} \left[\sup\limits_{r \in [s,T]}|\hat{X}(r)|^{2\bar{q}^2} \left(\int_s^T \kappa(r)^2dr\right)^{\bar{q}^2}\left(\int_s^T(1+|q(r)|)^2dr\right)^{\bar{q}^2}\Big|\mathcal{F}_s^t \right]\right)^{\frac{1}{\bar{q}^2}}\\
	&\leq C\left(\mathbb{E} \left[\sup\limits_{r \in [s,T]}|\hat{X}(r)|^{4\bar{q}^2}|\mathcal{F}_s^t \right]\right)^{\frac{1}{2\bar{q}^2}} 
     \left(\mathbb{E} \left[\left(\int_s^T \kappa(r)^2dr\right)^{4\bar{q}^2}\Big|\mathcal{F}_s^t \right]\right)^{\frac{1}{4\bar{q}^2}}\\
	&\quad\cdot \left(\mathbb{E} \left[\left(\int_s^T(1+|q(r)|)^2dr\right)^{4\bar{q}^2}\Big|\mathcal{F}_s^t \right]\right)^{\frac{1}{4\bar{q}^2}}\\
	&\leq\delta\left(|x'-\bar{X}^{t,x;\bar{u}}(s)|^4\right),
\end{aligned}
\end{equation}
and the estimate for $\left(\mathbb{E} \left[ \left(\int_s^T\rho(r)|\kappa(r)\hat{X}(r)|dr\right)^{2\bar{q}^2}\Big|\mathcal{F}_s^t \right]\right)^{\frac{1}{\bar{q}^2}}$ is similar to \eqref{estimate3}. 
By \eqref{the estimate of stochastic Lipschitz BSDE1}, we have
\begin{equation}\label{estimate4}
	\left(\mathbb{E} \left[ \left(\int_s^T|\kappa(r)|^2dr\right)^{2\bar{q}^2}\Big|\mathcal{F}_s^t \right]\right)^{\frac{1}{\bar{q}^2}}\leq\delta\left(|x'-\bar{X}^{t,x;\bar{u}}(s)|^4\right).
\end{equation}
And then we prove the last one.
\begin{equation}\label{estimate5}
\begin{aligned}
	&\left(\mathbb{E} \left[ \left(\int_s^T(1+|P(r)|)\rho(r)^2|\hat{X}(r)|^2dr\right)^{2\bar{q}^2}\Big|\mathcal{F}_s^t \right]\right)^{\frac{1}{\bar{q}^2}}\\
	&\leq C\left(\mathbb{E} \left[\left(1+\sup\limits_{r \in [s,T]}|P(r)|^{2\bar{q}^2}\right)\sup\limits_{r \in [s,T]}|\hat{X}(r)|^{4\bar{q}^2} \left(\int_s^T\rho(r)^2dr\right)^{2\bar{q}^2}\Big|\mathcal{F}_s^t \right]\right)^{\frac{1}{\bar{q}^2}}\\
	&\leq\delta\left(|x'-\bar{X}^{t,x;\bar{u}}(s)|^4\right).
\end{aligned}
\end{equation}
Thus, we obtain
\begin{equation}\label{stochastic Lipschitz BSDE2}
	|\hat{\gamma}(s)|\leq\delta(|x'-\bar{X}^{t,x;\bar{u}}(s)|^{2}), \quad \mathbf{P}\text{-a.s.}.
\end{equation}

Step 5. Completion of the proof.

Since the set of all rational vectors $x' \in \mathbb{R}^{n}$ is countable, we can find a subset $\Omega_0 \subseteq \Omega$ with $\mathbf{P}(\Omega_0) = 1$ such that for any $\omega_0 \in\Omega_0$,
\begin{equation*}
\left\{
\begin{aligned}
	&V(s,\bar{X}^{t,x;\bar{u}}(s,\omega_0)) = \bar{Y}^{t,x;\bar{u}}(s,\omega_0),\ \eqref{estimate of recursive FBSDE1},\eqref{eatimates of remainder terms of FBSDE},\eqref{relationship between X,Y and Z},
     \eqref{the estimate of stochastic Lipschitz BSDE1},\eqref{the estimate of stochastic Lipschitz BSDE2}\ \text{are satisfied for} \\
	&\text{any rational vector $x'$} ,(\Omega,\mathcal{F},\mathbf{P}(\cdot|\mathcal{F}_s^t)(\omega_0),W(\cdot)-W(s)) \in \mathcal{U}^{\omega}[s,T],\\
	&\text{and} \sup_{s\leq r\leq T} \left[|p(r,\omega_0)| + |P(r,\omega_0)|\right] < \infty.
\end{aligned}
\right.
\end{equation*}
The first relation of the above is obtained by the DPP in \cite{M21}. Let $\omega_0 \in \Omega_0$ be fixed, and then for any rational vector $x' \in \mathbb{R}^n$, by \eqref{stochastic Lipschitz BSDE2}, we have
\begin{equation*}
	|\hat{\gamma}(s,\omega_0)|\leq \delta\left(\left|x'-\bar{X}^{t,x;\bar{u}}{(s,\omega_0)}\right|^{2}\right),\quad \text{for all}\ s\in[t,T].
\end{equation*}
By the definition of $\hat{\gamma}(s)$, we have for all $s\in[t,T]$,
\begin{equation*}
\begin{aligned}
	&Y^{s,x',\bar{u}}{(s,\omega_0)} - \bar{Y}^{t,x,\bar{u}}{(s,\omega_0)} \leq \left\langle p(s,\omega_0), \left(x' - \bar{X}^{t,x;\bar{u}}(s,\omega_0)\right) \right\rangle \\
	&\quad + \frac{1}{2} \left\langle P(s,\omega_0)\left(x' - \bar{X}^{t,x;\bar{u}}(s,\omega_0)\right), \left(x' - \bar{X}^{t,x;\bar{u}}{(s,\omega_0)}\right) \right\rangle +\delta\left( \left|x' - \bar{X}^{t,x;\bar{u}}(s,\omega_0) \right|^2 \right).
\end{aligned}
\end{equation*}
Thus for all $s\in[t,T]$,
\begin{equation}\label{the different of the value function}
\begin{aligned}
	V&(s,x')-V(s,\bar{X}^{t,x;\bar{u}}(s,\omega_0))\leq Y^{s,x';\bar{u}}(s,\omega_0) - \bar{Y}^{t,x;\bar{u}}(s,\omega_0) 
	 \leq \left\langle p(s,\omega_0), \left(x' - \bar{X}^{t,x;\bar{u}}(s,\omega_0)\right) \right\rangle \\
	&\quad + \frac{1}{2} \left\langle P(s,\omega_0)\left(x' - \bar{X}^{t,x;\bar{u}}(s,\omega_0)\right), \left(x' - \bar{X}^{t,x;\bar{u}}(s,\omega_0)\right) \right\rangle + \delta \left( \left| x' - \bar{X}^{t,x;\bar{u}}(s,\omega_0) \right|^2 \right).
\end{aligned}
\end{equation}
Note that the term $\delta(|x' - \bar{X}^{t,x;\bar{u}}{(s,\omega_0)}|^2)$ in the above depends only on the size of $|x' - \bar{X}^{t,x;\bar{u}}{(s,\omega_0)}|$ and it is independent of $x'$. 
Therefore, by the continuity of $V(s, \cdot)$, we see that \eqref{the different of the value function} holds for all $x' \in \mathbb{R}^n$ (for more similar details see \cite{Z90-2}), 
which by the definition \eqref{second-order super- and sub-jets in the spatial variable} proves
\begin{equation*}
	\{p(s)\} \times [P(s),\infty) \subseteq D_x^{2,+} V(s,\bar{X}^{t,x;\bar{u}}(s)),\ \forall s \in [t,T],\ \mathbf{P}\text{-a.s.}.
\end{equation*}

Finally, fix an $\omega \in \Omega$ such that \eqref{the different of the value function} holds for any $x' \in \mathbb{R}^n$. 
For any $(\hat{p}, \hat{P}) \in D_x^{2,-} V\left(s, \bar{X}^{t,x;\bar{u}}(s)\right)$, by  definition \eqref{second-order super- and sub-jets in the spatial variable} we have
\small
\begin{equation*}
\begin{aligned}
	0 &\leq \liminf_{x' \to \bar{X}^{t,x;\bar{u}}(s)} \left\{ \frac{V(s, x') - V\left(s, \bar{X}^{t,x;\bar{u}}(s)\right) 
     - \left\langle \hat{p}, x' - \bar{X}^{t,x;\bar{u}}(s) \right\rangle-\frac{1}{2} \left\langle \hat{P}(x' - \bar{X}^{t,x;\bar{u}}(s)), x' - \bar{X}^{t,x;\bar{u}}(s) \right\rangle}{|x' - \bar{X}^{t,x;\bar{u}}(s)|^2} \right\} \\
	&\leq \liminf_{x' \to \bar{X}^{t,x;\bar{u}}(s)} \left\{ \frac{\left\langle p(s) - \hat{p}, x' - \bar{X}^{t,x;\bar{u}}(s) \right\rangle 
     + \frac{1}{2} \left\langle (P(s) - \hat{P})(x' - \bar{X}^{t,x;\bar{u}}(s)), x' - \bar{X}^{t,x;\bar{u}}(s) \right\rangle}{|x' - \bar{X}^{t,x;\bar{u}}(s)|^2} \right\}.
\end{aligned}
\end{equation*}
\normalsize
Then, it is necessary that
\begin{equation*}
	\hat{p} = p(s),\ \hat{P} \leq P(s),\quad \forall s \in [t,T],\ \mathbf{P}\text{-a.s.}.
\end{equation*} 
Thus, \eqref{relation of adjoint and second-order super and sub-jets in spatial variable} holds. \eqref{relation of adjoint and first-order super and sub-jets in spatial variable} is immediate. The proof is complete.
\end{proof}

\subsection{Proof of Theorem \ref{relation of MP and DPP in time variable}}

\begin{proof}
The proof is divided into two steps.

Step 1: Variations and estimations for FBSDE.

For any $s \in (t,T]$, choose $\tau \in (s, T]$. 
Denote by $(X^{\tau, \bar{X}^{t,x; \bar{u}}(s), \bar{u}}(\cdot),Y^{\tau, \bar{X}^{t,x; \bar{u}}(s), \bar{u}}(\cdot),Z^{\tau, \bar{X}^{t,x; \bar{u}}(s), \bar{u}}(\cdot))$ the solution to the following FBSDE on $ [\tau, T] $:
\begin{equation*}
\begin{cases}
\begin{aligned}
	 X^{\tau, \bar{X}^{t,x; \bar{u}}(s); \bar{u}}(r) &= \bar{X}^{t,x; \bar{u}}(s) + \int_{\tau}^{r} b( \alpha, X^{\tau, \bar{X}^{t,x; \bar{u}}(s), \bar{u}}(\alpha), \bar{u}(\alpha) ) d\alpha \\
	&\quad+ \int_{\tau}^{r} \sigma ( \alpha, X^{\tau, \bar{X}^{t,x; \bar{u}}(s), \bar{u}}(\alpha), \bar{u}(\alpha) ) dW(\alpha),\\
	 Y^{\tau, \bar{X}^{t,x; \bar{u}}(s); \bar{u}}(r) &= h(X^{\tau, \bar{X}^{t,x; \bar{u}}(s); \bar{u}}(T)) +\int_{r}^{T} \left[f( \alpha, X^{\tau, \bar{X}^{t,x; \bar{u}}(s), \bar{u}}(\alpha), \bar{u}(\alpha))\right. \\
	&\quad +\left. \frac{\mu}{2}|Z^{\tau, \bar{X}^{t,x; \bar{u}}(s); \bar{u}}(\alpha)|^2\right]d\alpha -  \int_{r}^{T}Z^{\tau, \bar{X}^{t,x; \bar{u}}(s); \bar{u}}(\alpha)dW(\alpha).
\end{aligned}
\end{cases}
\end{equation*}
For $r \in [\tau, T] $, Set
\begin{equation*}
\begin{aligned}
	 &\hat{X}_{\tau}(r) := X^{\tau, \bar{X}^{t,x; \bar{u}}(s); \bar{u}}(r) -\bar{X}^{t,x; \bar{u}}(r),\ \hat{Y}_{\tau}(r) := Y^{\tau, \bar{X}^{t,x; \bar{u}}(s); \bar{u}}(r) -\bar{Y}^{t,x; \bar{u}}(r),\\
	 &\hat{Z}_{\tau}(r) := Z^{\tau, \bar{X}^{t,x; \bar{u}}(s); \bar{u}}(r) -\bar{Z}^{t,x; \bar{u}}(r).
\end{aligned}
\end{equation*}
By Lemma 2.2 and 2.3, we have the following estimates for any $p \geq 2$,
\begin{equation}\label{estimations for variations FBSDE1}
\begin{aligned}
	&\mathbb{E} \left[ \sup_{\tau \leq r \leq T}\Big(|\hat{X}_{\tau}(r)|^{p}+ |\hat{Y}_{\tau}(r)|^{p} \Big) + \Big(\int_{\tau}^{T}|\hat{Z}_{\tau}(r)|^2dr\Big)^{\frac{p}{2}}\Big| \mathcal{F}_{\tau}^t \right] \\
	&\leq C |\bar{X}^{t,x;\bar{u}}(\tau) - \bar{X}^{t,x;\bar{u}}(s)|^{p}, \quad \mathbf{P}\text{-a.s.}.
\end{aligned}
\end{equation}
Note that
\begin{equation*}
	\bar{X}^{t,x;\bar{u}}(\tau) - \bar{X}^{t,x;\bar{u}}(s) = \int_{s}^{\tau} \bar{b}(r)dr + \int_{s}^{\tau} \bar{\sigma}(r)dW(r).
\end{equation*}
Taking conditional expectation $\mathbb{E}\left[\cdot|\mathcal{F}_s^t\right]$ on both sides of \eqref{estimations for variations FBSDE1}, we get for \text{a.e.} $s \in [t,T]$,
\begin{equation}\label{estimate of FBSDE}
	\mathbb{E} \left[ \sup_{\tau \leq r \leq T}\Big(|\hat{X}_{\tau}(r)|^{p}+ |\hat{Y}_{\tau}(r)|^{p} \Big) + \Big(\int_{\tau}^{T}|\hat{Z}_{\tau}(r)|^2dr\Big)^{\frac{p}{2}}\Big| \mathcal{F}_{s}^t \right] 
	\leq C(|\tau-s|^{\frac{p}{2}}), \quad \mathbf{P}\text{-a.s.},
\end{equation}
as $\tau \downarrow  s$. The process $\hat{X}_{\tau}(\cdot), \hat{Y}_{\tau}(\cdot) $ and $ \hat{Z}_{\tau}(\cdot)$ satisfies the following variational equations:
\begin{equation}\label{variational equations 2}
\left\{
\begin{aligned}
	d\hat{X}_{\tau}(r) &= \left[ \bar{b}_x(r) \hat{X}_{\tau}(r) + \varepsilon_{\tau 1}(r) \right] dr + \left[ \bar{\sigma}_x(r) \hat{X}_{\tau}(r) + \varepsilon_{\tau 2}(r) \right] dW(r), \\  
	\hat{X}_{\tau}(\tau) &= -\int_{s}^{\tau} \bar{b}(r) dr - \int_{s}^{\tau} \bar{\sigma}(r) dW(r),\\
	d\hat{Y}_{\tau}(r) &= -\left[ \bar{f}_x(r) \hat{X}_{\tau}(r) + \mu\bar{Z}^{t,x;\bar{u}}(r)\hat{Z}_{\tau}(r) +\varepsilon_{\tau 3}(r) \right] dr + \hat{Z}_{\tau}(r)dW(r), \\  
	\hat{Y}_{\tau}(T) &= h_{x}(\bar{X}^{t,x;\bar{u}}(T))\hat{X}_{\tau}(T) +\varepsilon_{\tau 4}(T) ,
\end{aligned}
\right.
\end{equation}
where
\begin{equation*}
\begin{aligned}
	\varepsilon_{\tau 1}(r) &:= \int_{0}^{1} \left[ b_x \left( r, \bar{X}^{t,x;\bar{u}}(r) + \theta \hat{X}_{\tau}(r), \bar{u}(r) \right) - \bar{b}_x(r) \right] \hat{X}_{\tau}(r) d\theta, \\
	\varepsilon_{\tau 2}(r) &:= \int_{0}^{1} \left[ \sigma_x \left( r, \bar{X}^{t,x;\bar{u}}(r) + \theta \hat{X}_{\tau}(r), \bar{u}(r)  \right) - \bar{\sigma}_x(r) \right] \hat{X}_{\tau}(r) d\theta,\\ 
	\varepsilon_{\tau 3}(r) &:= \int_{0}^{1} \left[ f_x \left(  r, \bar{X}^{t,x;\bar{u}}(r) + \theta \hat{X}_{\tau}(r), \bar{u}(r)  \right) - \bar{f}_x(r) \right] \hat{X}_{\tau}(r)d\theta + \frac{\mu}{2}|\hat{Z}_{\tau}(r)|^2,\\
	\varepsilon_{\tau 4}(T) &:= \int_{0}^{1} \left[ h_x \left( \bar{X}^{t,x;\bar{u}}(T) + \theta \hat{X}_{\tau}(T) \right) - h_{x}(\bar{X}^{t,x;\bar{u}}(T)) \right] \hat{X}_{\tau}(T) d\theta.
\end{aligned}
\end{equation*}
Similar to the proof in Theorem 4.1, we obtain
\begin{equation*}
	Y^{\tau,\bar{X}^{t,x;\bar{u}}(s);\bar{u}}(\tau)-\bar{Y}^{t,x;\bar{u}}(\tau)\leq p(\tau)^{\top}\hat{X}_{\tau}(\tau)+\frac{1}{2}\hat{X}_{\tau}(\tau)^{\top}P(\tau)\hat{X}_{\tau}(\tau)+\delta(|\hat{X}_{\tau}(\tau)|^{2}), \quad \mathbf{P}\text{-a.s.},
\end{equation*}
where $\delta(\cdot): [0,\infty)\to [0,\infty)$ is a deterministic, continuous and increasing function, independent of $x' \in \mathbb{R}^n$, with $\frac{\delta(r)}{r}\to0$ as $r\to0$, here and below, $\delta$ may change from line to line. This implies for \text{a.e.} $ s \in [t,T] $,
\begin{equation*}
\begin{aligned}
	&\mathbb{E}\left[ Y^{\tau,\bar{X}^{t,x;\bar{u}}(s);\bar{u}}(\tau)-\bar{Y}^{t,x;\bar{u}}(\tau) \big| \mathcal{F}_{s}^{t} \right]\\
    &\leq\mathbb{E}\left[ p(\tau)^{\top}\hat{X}_{\tau}(\tau)+\frac{1}{2}\hat{X}_{\tau}(\tau)^{\top}P(\tau)\hat{X}_{\tau}(\tau) \Big| \mathcal{F}_{s}^{t} \right]+\delta\left(|\tau-s|\right),\quad \mathbf{P}\text{-a.s. as }\tau\downarrow s.
\end{aligned}
\end{equation*}

Step 2: Completion of the proof.

Note that $\left(\Omega,\mathcal{F},\mathbf{P}\left(\cdot|\mathcal{F}^{t}_{\tau}\right),W(\cdot)-W(\tau);u(\cdot)|_{[\tau,T]}\right)\in\mathcal{U}^{w}[\tau,T]$, $\mathbf{P}$-a.s.. 
Thus by the definition of value function $V$, we have
\begin{equation*}
	V(\tau,\bar{X}^{t,x;\bar{u}}(s)) \leq Y^{\tau,\bar{X}^{t,x;\bar{u}}(s);\bar{u}}(\tau),\quad \mathbf{P}\text{-a.s.}.
\end{equation*}
Taking $\mathbb{E}\left[\cdot\mid\mathcal{F}_{s}^{t}\right]$ on both sides and noting that $\bar{X}^{t,x;\bar{u}}(s)$ is $\mathcal{F}_{s}^{t}$-measurable, we have
\begin{equation}\label{the value function and BSDE}
	V(\tau,\bar{X}^{t,x;\bar{u}}(s)) \leq \mathbb{E}\left[Y^{\tau,\bar{X}^{t,x;\bar{u}}(s);\bar{u}}(\tau)|\mathcal{F}_s^t\right],\quad \mathbf{P}\text{-a.s.}.
\end{equation}
By the DPP of \cite{M21}, choose a common suset  $\Omega_0 \subseteq \Omega$ with $\mathbf{P}(\Omega_0) = 1$ such that for any $\omega_0 \in\Omega_0$, the following holdes:
\begin{equation*}
\left\{
\begin{aligned}
	&V(s,\bar{X}^{t,x;\bar{u}}(s,\omega_0)) = \bar{Y}^{t,x;\bar{u}}(s,\omega_0),\ \eqref{estimate of FBSDE},\eqref{the value function and BSDE}\ \text{hold for any rational number $\tau > s$}, \\
	&(\Omega,\mathcal{F},\mathbf{P}(\cdot|\mathcal{F}_s^t)(\omega_0),W_\cdot-W_s;u_\cdot|_{[s,T]}) \in \mathcal{U}^{\omega}[s,T], \text{ and} \sup_{s\leq r\leq T} \left[|p(r,\omega_0)| + |P(r,\omega_0)|\right] < \infty.
\end{aligned}
\right.
\end{equation*}
Let $\omega_0 \in\Omega_0$ be fixed. Then for any rational number $\tau > s$ and for \text{a.e.} $s \in [t,T)$, we have
\begin{equation}\label{the differece of the value function}
\begin{aligned}
	V&(\tau,\bar{X}^{t,x;\bar{u}}(s,\omega_0)) - V(s,\bar{X}^{t,x;\bar{u}}(s,\omega_0)) \leq \mathbb{E}\left[Y^{\tau,\bar{X}^{t,x;\bar{u}}(s);\bar{u}}(\tau) - \bar{Y}^{t,x;\bar{u}}(s) \big| \mathcal{F}_s^t \right](\omega_0)\\
	&= \mathbb{E}\left[Y^{\tau,\bar{X}^{t,x;\bar{u}}(s);\bar{u}}(\tau) -\bar{Y}^{t,x;\bar{u}}(\tau) +\bar{Y}^{t,x;\bar{u}}(\tau) -\bar{Y}^{t,x;\bar{u}}(s) \big| \mathcal{F}_s^t \right](\omega_0)\\
	&\leq \mathbb{E}\left[ \langle p(\tau),\hat{X}_\tau(\tau) \rangle + \frac{1}{2}\langle P(\tau)\hat{X}_\tau(\tau),\hat{X}_\tau(\tau) \Big| \mathcal{F}_s^t \right](\omega_0)\\
	&\quad +\mathbb{E}\left[ \bar{Y}^{t,x;\bar{u}}(\tau) -\bar{Y}^{t,x;\bar{u}}(s) | \mathcal{F}_s^t \right](\omega_0) + \delta(|\tau-s|) \\
	&=\mathbb{E}\left[-\int_{s}^{\tau}\left[\bar{f}(r)+\frac{\mu}{2}|\bar{Z}^{t,x;\bar{u}}(r)|^{2}\right]dr + \langle p(\tau),\hat{X}_\tau(\tau) \rangle \right.\\[-4pt]
    &\qquad + \left.\frac{1}{2}\langle P(\tau)\hat{X}_\tau(\tau),\hat{X}_\tau(\tau) \Big| \mathcal{F}_s^t \right](\omega_0)+ \delta(|\tau-s|).
\end{aligned}
\end{equation}

Now let us estimate the terms on the right-hand side of \eqref{the differece of the value function}. To this end, we first note that for any noting that \(\omega_0 \in \Omega_0\) is fixed, thus for any square-integrable functions \(\alpha(\cdot), \hat{\alpha}(\cdot), \beta(\cdot) \in \mathcal{M}_{\mathcal{F}}^2([0, T]; \mathbb{R}^n)\), we have
\begin{equation*}
\begin{aligned}
	\mathbb{E} &\left[ \left\langle \int_s^{\tau} \alpha(r) dr, \int_s^{\tau} \hat{\alpha}(r) dr \right\rangle \Big| \mathcal{F}_s^t \right] (\omega_0)\\
	&\leq \left\{ \mathbb{E} \left[ \left| \int_s^\tau \alpha(r) dr \right|^2 \Big| \mathcal{F}_s^t \right] (\omega_0) \right\}^{\frac{1}{2}} 
     \cdot \left\{ \mathbb{E} \left[ \left| \int_s^\tau \hat{\alpha}(r) dr \right|^2 \Big| \mathcal{F}_s^t \right] (\omega_0) \right\}^{\frac{1}{2}}\\
	&\leq (\tau - s) \left\{ \int_s^\tau \mathbb{E} \left[ |\alpha(r)|^2 \middle| \mathcal{F}_s^t \right] (\omega_0) dr \cdot \int_s^\tau \mathbb{E} \left[ |\hat{\alpha}(r)|^2 \middle| \mathcal{F}_s^t \right] (\omega_0) dr \right\}^{\frac{1}{2}}\\
    &=\delta(|\tau-s|), \quad \text{as } \tau \downarrow s, \quad \text{a.e.} s \in [t, T),
\end{aligned}
\end{equation*}
and
\begin{equation*}
\begin{aligned}
	\mathbb{E} &\left[ \left\langle \int_s^\tau \alpha(r) dr, \int_s^\tau \beta(r) dW(r) \right\rangle \Big| \mathcal{F}_s^t \right] (\omega_0)\\
	&\leq \left\{ \mathbb{E} \left[ \left| \int_s^\tau \alpha(r) dr \right|^2 \Big| \mathcal{F}_s^t \right] (\omega_0) \right\}^{\frac{1}{2}} 
     \cdot \left\{ \mathbb{E} \left[ \left| \int_s^\tau \beta(r) dW(r) \right|^2 \Big| \mathcal{F}_s^t \right] (\omega_0) \right\}^{\frac{1}{2}}\\
	&\leq (\tau - s)^{\frac{1}{2}} \left\{ \int_s^\tau \mathbb{E} \left[ |\alpha(r)|^2 \middle| \mathcal{F}_s^t \right] (\omega_0) dr \cdot \int_s^\tau \mathbb{E} \left[ |\beta(r)|^2 \middle| \mathcal{F}_s^t \right] (\omega_0) dr \right\}^{\frac{1}{2}}\\
	&=\delta(|\tau-s|), \quad \text{as } \tau \downarrow s, \quad \text{a.e.} s \in [t, T).
\end{aligned}
\end{equation*}
Each last inequality in the above two inequalities holds, since the sets of right Lebesgue points have full Lebesgue measures for integrable functions, and $s\mapsto \mathcal{F}_s^t$ is right continuous in $s$. 
Thus by \eqref{variational equations 2} and \eqref{the first-order adjoint equation}, we obtain,
\begin{equation}\label{estimate 1 of adjoint and SDE}
\begin{aligned}
	&\mathbb{E}\left[ \langle p(\tau),\hat{X}_\tau(\tau) \rangle| \mathcal{F}_s^t\right](\omega_0) 
	= \mathbb{E}\left[ \langle p(s),\hat{X}_\tau(\tau) \rangle -  \langle (p(\tau)-p(s)),\hat{X}_\tau(\tau) \rangle| \mathcal{F}_s^t\right](\omega_0)\\
	&=\mathbb{E}\left[ \left\langle p(s), -\int_{s}^{\tau} \bar{b}(r)dr-\int_{s}^{\tau} \bar{\sigma}(r)dW(r)\right\rangle \right.\\
	&\qquad+\left\langle -\int_{s}^{\tau}\bigg\{\bar{b}_x(r)^{\top}p(r) + \bar{f}_x(r) + \bar{\sigma}_x(r)^{\top}q(r) + \mu \bar{Z}^{t,x;\bar{u}}(r)\left(\bar{\sigma}_x(r)^{\top}p(r) + q(r)\right) \bigg\}dr\right.\\
	&\qquad\quad \left. +\left.\int_{s}^{\tau}q(r)dW(r),-\int_{s}^{\tau} \bar{b}(r)dr-\int_{s}^{\tau} \bar{\sigma}(r)dW(r)\right\rangle \Big| \mathcal{F}_s^t\right](\omega_0)\\
	&\leq\mathbb{E}\left[ \langle p(s), -\int_{s}^{\tau} \bar{b}(r)dr\rangle-\int_{s}^{\tau}\langle q(r),  \bar{\sigma}(r)\rangle dr \Big| \mathcal{F}_s^t\right](\omega_0) + \delta(|\tau-s|).
\end{aligned}
\end{equation}
Similarly, by \eqref{variational equations 2} and \eqref{the second-order adjoint equation}, we obtain
\begin{equation}\label{estimate 2 of adjoint and SDE}
\begin{aligned}
	&\mathbb{E}\left[\hat{X}_\tau(\tau)^{\top}P(\tau)\hat{X}_\tau(\tau)| \mathcal{F}_s^t\right](\omega_0)\\
	&=\mathbb{E}\left[ \hat{X}_\tau(\tau)^{\top}P(s)\hat{X}_\tau(\tau) + \hat{X}_\tau(\tau)^{\top}(P(\tau)-P(s))\hat{X}_\tau(\tau)| \mathcal{F}_s^t\right](\omega_0)\\
	&\leq \mathbb{E}\left[ \int_{s}^{\tau}\langle P(s)\bar{\sigma}(r),\bar{\sigma}(r)\rangle dr \Big| \mathcal{F}_s^t\right](\omega_0)+ \delta(|\tau-s|).
\end{aligned}
\end{equation}
Thus, by \eqref{the differece of the value function}-\eqref{estimate 2 of adjoint and SDE}, we have for any rational $\tau \in (s,T]$ and at $\omega=\omega_0$,
\begin{equation}\label{the estimate of the value function}
\begin{aligned}
	V&(\tau,\bar{X}^{t,x;\bar{u}}(s)) - V(s,\bar{X}^{t,x;\bar{u}}(s)) \\
	&\leq\mathbb{E}\bigg[-\int_{s}^{\tau}\left[\bar{f}(r)+\frac{\mu}{2}|\bar{Z}^{t,x;\bar{u}}(r)|^{2}\right]dr + \left\langle p(s), -\int_{s}^{\tau} \bar{b}(r)dr\right\rangle-\int_{s}^{\tau}\langle q(r),  \bar{\sigma}(r)\rangle dr\\
	&\qquad+ \frac{1}{2}\int_{s}^{\tau}\langle P(s)\bar{\sigma}(r),\bar{\sigma}(r)\rangle dr\Big| \mathcal{F}_s^t \bigg] + \delta(|\tau-s|)\\
	&=-(\tau-s)\mathcal{H}_{1}(s,\bar{X}^{t,x;\bar{u}}(s),\bar{Z}^{t,x;\bar{u}}(s),\bar{u}(s))+ \delta(|\tau-s|) .
\end{aligned}
\end{equation}
By definition \eqref{right super- and sub-jets in the time variable}, we obtain that the first relation of \eqref{relationship of adjoint and super- and sub-jets in time variable} holds for any (not only rational numbers) $\tau\in(s,T]$. 
Finally, fix an $\omega\in\Omega$ such that \eqref{the estimate of the value function} holds for any $\tau\in(s,T]$. 
Then for any $\hat{q} \in D_{t+}^{1,-}V(s,\bar{X}^{t,x;\bar{u}}(s))$, by definition \eqref{right super- and sub-jets in the time variable} and \eqref{the estimate of the value function} we have
\begin{equation*}
\begin{aligned}
	0 &\leq \liminf_{\tau\downarrow s} \left\{ \frac{V\left(\tau,\bar{X}^{t,x;\bar{u}}(s)\right) - V\left(s,\bar{X}^{t,x;\bar{u}}(s)\right) - \hat{q}(\tau-s)}{|\tau-s|} \right\}\\
	&\leq\liminf_{\tau\downarrow s} \left\{ \frac{\left( -\mathcal{H}_{1}(s,\bar{X}^{t,x;\bar{u}}(s),\bar{Z}^{t,x;\bar{u}}(s),\bar{u}(s)) - \hat{q} \right)(\tau-s)}{|\tau-s|} \right\}\\
	&=\liminf_{\tau\downarrow s} \left\{ \left( -\mathcal{H}_{1}(s,\bar{X}^{t,x;\bar{u}}(s),\bar{Z}^{t,x;\bar{u}}(s),\bar{u}(s)) - \hat{q} \right)\right\}	.
\end{aligned}
\end{equation*}
Thus\vspace{-3pt}
\begin{equation*}
	\hat{q} \leq -\mathcal{H}_{1}(s,\bar{X}^{t,x;\bar{u}}(s),\bar{Z}^{t,x;\bar{u}}(s),\bar{u}(s)).
\end{equation*}
Thus, the second relation of \eqref{relationship of adjoint and super- and sub-jets in time variable} holds. The proof is complete.
\end{proof}

\end{document}